\RequirePackage{fix-cm}
\documentclass[smallextended]{svjour3}       
\smartqed  
\usepackage{geometry}
\usepackage{graphicx}
\usepackage{amsmath}
\usepackage{amssymb}
\usepackage{bm}
\usepackage{subfigure}
\usepackage{diagbox}
\usepackage{mathrsfs}
\usepackage{multirow}
\usepackage{color}
\usepackage{url}
\usepackage{algorithm,algorithmic}
\makeatletter
\newenvironment{breakablealgorithm}
  {
   \begin{center}
     \refstepcounter{algorithm}
     \hrule height.8pt depth0pt \kern2pt
     \renewcommand{\caption}[2][\relax]{
       {\raggedright\textbf{\ALG@name~\thealgorithm} ##2\par}%
       \ifx\relax##1\relax 
         \addcontentsline{loa}{algorithm}{\protect\numberline{\thealgorithm}##2}%
       \else 
         \addcontentsline{loa}{algorithm}{\protect\numberline{\thealgorithm}##1}%
       \fi
       \kern2pt\hrule\kern2pt
     }
  }{
     \kern2pt\hrule\relax
   \end{center}
  }

\RequirePackage{hyperref}[6.83]
\hypersetup{
  colorlinks=false,
  frenchlinks=false,
  pdfborder={0 0 0},
  naturalnames=false,
  hypertexnames=false,
  breaklinks
}
\numberwithin{equation}{section}

\DeclareMathOperator*{\argmin}{argmin}

\DeclareMathOperator*{\sign}{sign}

\def \bbR {\mathbb R}
\def \bbN {\mathbb N}

\def \bd {\bm{d}}

\def \bp {\bm{p}}
\def \bq {\bm{q}}
\def \bx {\bm{x}}
\def \by {\bm{y}}

\def \bw {\bm{w}}
\def \bu {\bm{u}}
\def \bv {\bm{v}}
\def \br {\bm{r}}

\def \bR {\bm{R}}

\def \bI {\bm{I}}
\def \bA {\bm{A}}
\def \bB {\bm{B}}
\def \bC {\bm{C}}

\def \bQ {\bm{Q}}
\def \bR {\bm{R}}
\def \bU {\bm{U}}
\def \bV {\bm{V}}

\def \bbRn {\bbR^{n}}

\def \bbRN {\bbR^{N}}

\def \tbQ {\tilde{\bQ}}

\def \prox {\mathrm{prox}}

\def \dist {\mathrm{dist}}

\def \dom {\mathrm{dom}\hspace{2pt}}

\newcommand\leqs{\leqslant}
\newcommand\geqs{\geqslant}

\begin{document}

\title{Globally optimal solutions to a class of fractional optimization problems based on proximal gradient algorithm
}

\titlerunning{Globally Optimal Solutions to Fractional Optimization Problems}        

\author{Yizun Lin     \and
        Jian-Feng Cai \and
        Zhao-Rong Lai \and
        Cheng Li
}


\institute{Yizun Lin\at
              School of Mathematics, Jinan University, Guangzhou 510632, China.\\
              \email{linyizun@jnu.edu.cn}           
           \and
           Jian-Feng Cai \at
              Department of Mathematics, The Hong Kong University of Science and Technology, Clear Water Bay, Kowloon, Hong Kong SAR, China\\
              \email{jfcai@ust.hk}
           \and
           Zhao-Rong Lai \at
              Corresponding author. Cyberspace Institute of Advanced Technology, Pei Dingyi Cryptography Research Institute, Guangzhou University, Guangzhou 510006, China\\
              \email{26072265@qq.com}
           \and
           Cheng Li \at
              School of Economics and Statistics, Guangzhou University, Guangzhou 510006, China.\\
              \email{lcyogurt1021@gmail.com}
}


\maketitle

\vspace{-5em}
\begin{abstract}
This paper investigates a category of constrained fractional optimization problems that emerge in various practical applications. The objective function for this category is characterized by the ratio of a numerator and denominator, both being convex, semi-algebraic and differentiable with Lipschitz continuous gradients over the constraint sets. The constrained sets associated with these problems are closed, convex, and semi-algebraic. We propose an efficient algorithm that is inspired by the proximal gradient method, and we provide a thorough convergence analysis. Our algorithm is shown to converge to a critical point without the typical need for a nonnegative numerator, and this critical point becomes a globally optimal solution whenever the numerator is nonpositive at the initial point. Our approach is adaptable to unbounded constraint sets as well. Therefore, our approach is viable for many more practical models. Numerical experiments show that our method not only reliably reaches ground-truth solutions in some model problems but also outperforms several existing methods in maximizing the Sharpe ratio with real-world financial data.

\vspace{1em}
\keywords{Fractional optimization \and global optimum \and proximal gradient algorithm \and Sharpe ratio}

\vspace{1em}
\subclass{90C26 \and 90C32 \and 65K05}
\end{abstract}

\newpage
\section{Introduction}\label{sec:intro}
Single-ratio fractional optimization aims to minimize the ratio of two functions subject to a given constraint set. It can be formulated as
\begin{equation}\label{mod:fracfg1}
\min_{\bx\in\Omega}\left\{\frac{f(\bx)}{g(\bx)}\right\},
\end{equation}
where $f:\bbRn\to\bbR$ and $g:\bbRn\to\bbR$ are two functions, and $\Omega\subset\bbRn$ is the constraint set within which this optimization takes place. Interested readers are referred to \cite{FPreview1,FPreview3,FPreview4} for reviews of this field. Fractional optimization has applications in various areas such as operations research \cite{FPapp7,SRlincom}, mathematical finance \cite{FPapp2,FPapp1}, signal processing \cite{FPapp3,FPapp6}, wireless communications \cite{FPapp4,FPapp5}, and beyond.

An interesting special case of \eqref{mod:fracfg1} is the concave-convex fractional problem, expressed as
\begin{equation}\label{eqn:fracopt-ccfp}
\max_{\bx\in\Omega}\left\{\frac{-f(\bx)}{g(\bx)}\right\},
\end{equation}
where the constraint set $\Omega$ is convex, closed and bounded; function $-f$ is continuous and concave on $\Omega$; and the function $g$ is continuous, convex, and strictly positive over $\Omega$. This problem can primarily be solved by Dinkelbach's method \cite{frdindelbach,frdindelbach2,FPreview2}, the details of which are presented in Section \ref{sec:dinkelbachapproach}. This method transforms the original fractional optimization model into a parametric subtractive form and then alternates between solving the subtractive concave maximization problem and updating the parameter in the subtractive model. Dinkelbach's method is intuitive and efficient in many settings and provides a global optimality guarantee when the associated parametric subproblems are solved globally. Nevertheless, under the above assumptions, several practical issues may arise. First, the concavity of the parametric maximization subproblem depends on the sign of the parameter, and it may be difficult to preserve this structure when no initial point satisfying $f(\bx^0)\leqs0$ is available. Second, even when this structure is preserved, each parametric subproblem still needs to be solved to global or sufficiently high accuracy, which may lead to additional inner computational cost. Third, the compactness of $\Omega$, ensured here by closedness and boundedness, is a convenient sufficient condition for attainment of the parametric subproblems, but it excludes some problems with unbounded feasible regions.

Other methods exist for other special cases of problem (\ref{mod:fracfg1}). For instance, Pang \cite{SRlincom} proposes optimizing the Sharpe ratio (SR) with $f(\bx)=-\bm{\mu}^\top \bx$ and $g(\bx)=\sqrt{\bx^\top\bV\bx}$ on a closed and bounded domain $\Omega$, where $\bm{\mu}\in\bbR^{n}$ and the positive definite matrix $\bm{V}\in\bbR^{n\times n}$ represent the expected return vector and its covariance matrix for $n$ assets in a financial market, respectively. While this problem meets the assumptions for Dinkelbach's method, the corresponding subtractive concave maximization problem lacks a closed-form solution. Instead of using Dinkelbach's method, Pang \cite{SRlincom} reformulated the SR maximization model as a linear complementarity problem and solved it using the principle pivoting algorithm \cite{SRprinpivot}. This method requires a vector $\tilde{\bx}\in \Omega$ such that $f(\tilde{\bx})<0$.

Bo{\c{t}} and Csetnek \cite{boct2017proximal} introduced a proximal gradient algorithm tailored to address model \eqref{mod:fracfg1} when the numerator is proper, lower semicontinuous, and convex, with a smooth denominator. Zhang and Li \cite{zhang2022first} proposed a proximal gradient subgradient algorithm to solve a class of problems \eqref{mod:fracfg1}, where $g$ is convex but may not necessarily smooth, and $f$ admits a decomposition $f=h_1+h_2$ such that $h_1:\bbRn\to[-\infty,+\infty]$ is proper, lower semicontinuous, bounded below on $\bbRn$, $h_2$ is Lipschitz differentiable. This method stipulates that $f(\bx)\geqs0$ for all $\bx\in\Omega$, and the level-boundedness of $f/g$, to guarantee convergence to a critical point of the subtractive form of the objective function. Li et al. \cite{li2022proximal} further developed a proximal-gradient-subgradient algorithm with backtracked extrapolation for structured fractional programming and established subsequential convergence to critical points, as well as global sequential convergence under additional Kurdyka-${\L}$ojasiewicz (KL)-type assumptions. Subsequently, Bo{\c{t}} et al. \cite{boct2022extrapolated} proposed a proximal subgradient algorithm that incorporates extrapolations for the single-ratio fractional optimization model with a proper and lower semicontinuous numerator and denominator, which are not necessarily convex. Extending their work, Bo{\c{t}} et al. \cite{boct2023inertial} developed an inertial proximal block coordinate method for solving nonsmooth sum-of-ratios fractional optimization problems with block structure. Overall, these proximal gradient or subgradient methods mainly focus on convergence to critical points of fractional optimization models. This naturally raises the question of whether globally optimal solutions can be obtained under suitable additional assumption.

In this paper, we present an efficient algorithm for a class of
problems \eqref{mod:fracfg1}, where both $f$ and $g$ are convex, semi-algebraic and differentiable with Lipschitz continuous gradients on the
constraint set $\Omega$, and $\Omega$ is closed, convex, and semi-algebraic. The contributions of our work are threefold:
\begin{enumerate}
\item We present a proximal gradient algorithm (PGA) for solving the considered fractional optimization model and establish convergence of the generated sequence to a critical point. Under the verifiable initial-point condition $f(\bx^0)\leqs0$, this critical point is further shown to be a globally optimal solution.

\item Our algorithm dispenses with the commonly required nonnegativity condition for the numerator $-f$ (or $f$), thereby accommodating a broader range of scenarios. While many algorithms necessitate a condition no less stringent than $-f({\bx})\geqs0$ \cite{FPreview1,frdindelbach2,FPreview2} or $f({\bx})\geqs0$ \cite{zhang2022first} for all ${\bx}\in \Omega$, the PGA can be performed without imposing a global sign condition on the numerator.

\item The boundedness of the constraint set is not imposed explicitly in our analysis. Instead, we use a level-boundedness-type condition on the fractional objective to ensure boundedness of the generated sequence and to support the convergence analysis. This allows the framework to include certain problems with unbounded feasible regions.
\end{enumerate}

The rest of this paper is organized as follows. Section \ref{sec:relatework} reviews several closely related works, underscoring the novelty of our algorithm. Section \ref{sec:PGAforFracOpt} provides a detailed explanation of our algorithm. Section \ref{sec:numerexam} demonstrates the application of our algorithm to practical tasks, including the prevalent and crucial task of Sharpe ratio maximization. Section \ref{sec:experiment} presents experimental results confirming that our algorithm not only achieves ground-truth optimal solutions but also excels in real-world optimization tasks. Finally, we conclude our paper in Section \ref{sec:conclusion}.

\section{Related works}\label{sec:relatework}
Interested readers may refer to review papers \cite{FPreview1,FPreview3,FPreview4} for more comprehensive information on the topic of fractional optimization. In this section, we discuss several works closely related to ours to better under score the motivations and contributions of our proposed method.

\subsection{Dinkelbach's approach}
\label{sec:dinkelbachapproach}
Dinkelbach's approach \cite{frdindelbach} to solve \eqref{eqn:fracopt-ccfp} consists of the following steps.
\begin{itemize}
\item[$Step\ 1.$] Choose an initial point $\bx^{0}\in \Omega$ and compute $c_{0}=\frac{-f(\bx^{0})}{g(\bx^{0})}$. Set $k=0$.
\item[$Step\ 2.$] Solve the concave maximization problem:
\begin{equation}
\label{eqn:fracopt-dink}
\max_{\bx\in\Omega} \left\{-f(\bx)-c_{k}g(\bx)  \right\}
\end{equation}
to obtain an optimal solution $\bx^{k+1}$.
\item[$Step\ 3.$] If $f(\bx^{k+1})+c_{k}g(\bx^{k+1})=0$, then define the optimal solution as $\bx^*:=\bx^{k+1}$ and stop.
\item[$Step\ 4.$] If $f(\bx^{k+1})+c_{k}g(\bx^{k+1})<0$, update $c_{k+1}=\frac{-f(\bx^{k+1})}{g(\bx^{k+1})}$, increment $k$ by one, and return to Step 2.
\end{itemize}
This approach primarily establishes a connection between (\ref{eqn:fracopt-ccfp}) and its parametric counterpart:
\begin{gather}
\label{eqn:fracopt-dink-para}
F(c):=\max_{\bx\in\Omega}\left\{-f(\bx)-cg(\bx)\right\},\\
\notag c^*=\frac{-f(\bx^*)}{g(\bx^*)}=\max_{\bx\in\Omega}\left\{\frac{-f(\bx)}{g(\bx)}\right\}\quad  \text{if and only if}\quad F(c^*)=0.
\end{gather}
Problem (\ref{eqn:fracopt-ccfp}) is thus equivalent to finding a number $c^*$ such that $F(c^*)=0$. It is crucial to maintain (\ref{eqn:fracopt-dink-para}) as a concave maximization model for $\bx$ to ensure that an optimal point is obtained numerically. For this purpose, it is sufficient to choose an initial point $\bx^{0}\in \Omega$ to satisfy $c_0\geqs 0$. Consequently, the sequence $\{c_k\}$  will be strictly monotonically increasing, preserving the concavity of the maximization objective function in \eqref{eqn:fracopt-dink-para}.

However, the concavity of the parametric maximization objective in
\eqref{eqn:fracopt-dink-para} relies on the nonnegativity of the parameter $c_k$. Under the ratio-based initialization $c_0=-f(\bx^0)/g(\bx^0)$, this property is guaranteed when $\bx^0\in\Omega$ satisfies $f(\bx^0)\leqs0$, whereas it is generally no longer ensured if $c_0<0$. Moreover, when $\Omega$ is unbounded, attainment of optimal solutions for \eqref{eqn:fracopt-ccfp} and the parametric subproblems in \eqref{eqn:fracopt-dink-para} is not automatic and may require additional assumptions. Additionally, even the concave maximization problem
\eqref{eqn:fracopt-dink-para} could incur high computational costs.

We emphasize that the above discussion is intended to clarify the scope and practical
requirements of Dinkelbach-type methods. Such methods can provide global optimality guarantees when the relevant optimal solutions exist and the associated parametric subproblems are solved globally, while this entails globally solving a parametric optimization subproblem at each outer iteration. In contrast, the proposed PGA avoids solving such a subproblem at each iteration and guarantees convergence to a critical point, which is further shown to be globally optimal under the verifiable initial-point condition $f(\bx^0)\leqs0$.

\subsection{Pang's approach}
\label{sec:LCP}
When (\ref{eqn:fracopt-ccfp}) assumes specific forms, it can be reformulated as problems other than (\ref{eqn:fracopt-dink-para}). For instance, in the case of Sharpe ratio optimization \cite{SRlincom},
\begin{gather*}
f(\bx):=-\bm{\mu}^\top \bx,\quad g(\bx):=\sqrt{\bx^\top\bV\bx},\\
\Omega:=\left\{\bx\in\bbRn|\hspace{2pt}\bx\geqs{\bm0}_n,\ \ \bx^\top{\bm1}_n=1\ \mbox{and}\ \bC\bx\leqs\bd\right\},
\end{gather*}
where ${\bm0}_n$ and ${\bm1}_n$ represent the $n$-dimensional zero and one vectors, respectively. Here, $\bC\in\bbR^{m\times n}$ and $\bd\in \bbR^{m}$ constitute a linear constraint for $\bx$. The aforementioned problem is equivalent to a linear complementarity problem:
\begin{equation*}
\begin{cases}
\bu=-\bm{\mu}+\bV\bx+(\bC^\top-{\bm1}_n\bd^\top)\by\geqs {\bm0}_n, \quad \bx\geqs {\bm0}_n,\\
\bv=-(\bC-\bd{\bm1}_n^\top)\bx\geqs {\bm0}_m, \quad \by\geqs {\bm0}_m,\\
\bu^\top\bx=\bv^\top\by=0.
\end{cases}
\end{equation*}
This can be efficiently solved using the principle pivoting algorithm \cite{SRprinpivot}, yielding a globally optimal solution. The approach necessitates that $\mu_i>0$ for some $i$, indicating that at least an asset exhibits a positive return.

Our work explores a more general form of fractional optimization that encompasses the Sharpe ratio optimization. Our method guarantees convergence to a critical point without the prerequisite of $\mu_i>0$, broadening its applicability to a variety of scenarios.

\subsection{Zhang and Li's approach}
\label{sec:nonconfracopt}
Recently, Zhang and Li \cite{zhang2022first} addressed a category of nonconvex fractional optimization problems:
\begin{equation}\label{eqn:fracopt-zhang}
\min_{\bx\in\Omega}\left\{\frac{h_1(\bx)+h_2(\bx)}{g(\bx)}\right\},
\end{equation}
where $h_1:\bbRn\to[-\infty,+\infty]$ is proper, lower semicontinuous, bounded below on $\bbRn$, $h_2$ is Lipschitz differentiable, $h_1+h_2$ is nonnegative over $\bbRn$, and $g:\bbRn\to\bbR$ is positive and convex. They proposed a proximal gradient subgradient algorithm, described by the iteration scheme:
\begin{equation*}
\begin{cases}
\by^{k+1}\in\partial g(\bx^{k})\\
c_k=\frac{h_1(\bx^{k})+h_2(\bx^{k})}{g(\bx^{k})}\\
\bx^{k+1}\in \prox_{\alpha_{k}h_1}\left(\bx^{k}-\alpha_{k}\nabla h_2(\bx^{k})+\alpha_{k}c_{k}\by^{k+1}\right)
\end{cases},
\end{equation*}
where $\{\alpha_k\}$ are algorithmic parameters. The symbols `$\partial$' and `$\prox$' denote the limiting-subdifferential and the proximity operator, respectively, with definitions provided later in Definition \ref{subdifferentials} $(ii)$ and \eqref{def_prox}. This algorithm is designed to find a critical point $\bx^*$ of the surrogate objective function with the subtractive form $h_1+h_2-c^*g$, where $c^*:=\frac{h_1(\bx^*)+h_2(\bx^*)}{g(\bx^*)}$. Given that the sign of $h_1+h_2$ is opposite to that in Dinkelbach's method, the resulting $\bx^*$ is not necessarily a local or global optimum of either the surrogate objective function or the original objective function defined by \eqref{eqn:fracopt-zhang}. In this paper, we examine a special class of fractional optimization problems (\ref{eqn:fracopt-zhang}), where we are able to secure globally optimal solutions under certain conditions.

\section{Proximal gradient algorithm for fractional optimization problems}\label{sec:PGAforFracOpt}
In this section, we develop a proximal gradient algorithm (PGA) for solving a specific class of fractional optimization problems. We will prove that, under certain assumptions, our proposed algorithm is capable of converging to a globally optimal solution of the nonconvex fractional optimization model.

To more effectively describe the model we intend to solve, we first revisit the concepts of semi-algebraic sets and functions as defined by Attouch et al. \cite{attouch2010proximal}. Throughout this paper, $\overline{\bbR}$ denotes the extended real number set $\bbR\cup\{+\infty\}$.
\begin{definition}[Semi-algebraic sets and functions]\label{def:semialgebra}
A subset $\mathcal{S}\subset\bbRn$ is called semi-algebraic if it can be expressed as $\mathcal{S}=\bigcup\limits_{j=1}^s\bigcap\limits_{i=1}^t\{\bx\in\bbR^n|\,p_{ij}(\bm{x})=0, q_{ij}(\bm{x})<0\}$, with $p_{ij}$ and $q_{ij}$ being real polynomial functions for each $i\in\bbN_t$ and $j\in\bbN_s$, given some positive integers $s,t\in\bbN_+$. A function $\psi: \mathcal{S}\to\overline{\bbR}$ is semi-algebraic if its graph, defined as $\{(\bm{x},y)\in\mathcal{S}\times\bbR|\,y=\psi(\bx)\}$, is a semi-algebraic subset of $\mathbb{R}^{n+1}$.
\end{definition}

This paper deals with a category of fractional optimization problems as depicted in \eqref{mod:fracfg1}, where $\Omega\subset\bbRn$ is a set that is closed, convex, and semi-algebraic. The functions $f:\bbRn\to\bbR$ and $g:\bbRn\to\bbR$ adhere to the subsequent two assumptions:\\
{\bf Assumption 1.} The functions $f$ and $g$ are both convex, semi-algebraic and differentiable with Lipschitz continuous gradients over $\Omega$. Furthermore, $g(\bx)>0$ for all $\bx\in\Omega$.\\
{\bf Assumption 2.} For any $d\in\left\{\frac{f(\bx)}{g(\bx)}\,\Big|\,\bx\in\Omega\right\}$, the level set $\left\{\bx\in\Omega\,\Big|\,\frac{f(\bx)}{g(\bx)}\leq d\right\}$ is bounded.

It is worth noting that Assumption 2 is less stringent than the bounded-constraint-set assumption which is commonly required by many existing methods. The former leverages the level-boundedness of the function to ease the constraint on the variable, while the latter imposes a direct constraint on the variable independent of the properties of the function. It is also straightforward to identify examples that fulfill the former assumption but not the latter, as discussed in Section \ref{sec:exampunbound}. Significantly, Assumption 2 is also less demanding than the general sense of level-boundedness --- it only needs to apply within the set $\left\{\frac{f(\bx)}{g(\bx)}\,\Big|\,\bx\in\Omega\right\}$, rather than across the entirety of $\bbR$.

\begin{proposition}\label{prop:fglbound}
Given a closed set $\Omega\subset\mathbb{R}^n$, and two continuous functions $f$ and $g$ on $\Omega$ that fulfill Assumption 2, with $g(\bx)>0$ for all $\bx\in\Omega$, there exists a real number $M$ such that $\frac{f(\bx)}{g(\bx)}\geq M$ for all $\bx\in\Omega$.
\end{proposition}
\begin{proof}
Choose any $\bx^0\in\Omega$ and set $d:=f(\bx^0)/g(\bx^0)$. By Assumption 2, the leverl set $E:=\{\bx\in\Omega\mid f(\bx)/g(\bx)\leqs d\}$ is bounded, and it is nonempty since $\bx^0\in E$. Moreover, the continuity of $f/g$ on $\Omega$ together with the fact $\Omega$ is closed implies that $E$ is closed. Hence, $E$ is compact. Thus, $f/g$ attains a minimum $M$ on $E$. For $\bx\in\Omega\setminus E$, we have $f(\bx)/g(\bx)>d\geqs M$. Therefore, $f(\bx)/g(\bx)\geqs M$ for all $\bx\in\Omega$.
\end{proof}

To solve \eqref{mod:fracfg1}, we first recast it as an equivalent fractional optimization model whose objective function is guaranteed to be nonnegative. The nonnegativity of the objective function is essential for the later convergence proof of our algorithm. Proposition \ref{prop:fglbound} ensures the existence of an $M \in \mathbb{R}$ such that $f(\bx) - Mg(\bx) \geq 0$. We then define the function $\tilde{f}:\bbR^n \to \bbR$ by
\begin{equation}\label{def:tildef}
\tilde{f}(\bx):=f(\bx)-M g(\bx),
\end{equation}
and denote $L_{\nabla f}$  and $L_{\nabla g}$ as the Lipschitz constants of $\nabla f$ and $\nabla g$, respectively. Then, the following facts are established:

\noindent{\bf Fact 1} $\tilde{f}(\bx)\geqs0$ for all $\bx\in\Omega$.

\noindent{\bf Fact 2} For any $d\in\left\{\frac{\tilde{f}(\bx)}{g(\bx)}\Big|\hspace{2pt}\bx\in\Omega\right\}$, the level set $\left\{\bx\in\Omega\Big|\hspace{2pt}\frac{\tilde{f}(\bx)}{g(\bx)}\leqs d\right\}$ is bounded.

\noindent{\bf Fact 3} $\nabla\tilde{f}$ is $L_{\nabla\tilde{f}}$-Lipschitz continuous on $\Omega$, where $L_{\nabla\tilde{f}}:=L_{\nabla f}+|M|\cdot L_{\nabla g}$.

We also define the indicator function $\iota_{\Omega}:\bbRn\to\overline{\bbR}$ by
\begin{equation}\label{def:iotaOmega}
\iota_{\Omega}(\bx):=\begin{cases}
0,&\mbox{if}\ \bx\in\Omega,\\
+\infty,&\mbox{otherwise}.
\end{cases}
\end{equation}
Note that $\iota_{\Omega}$ is lower semicontinuous and convex because $\Omega$ is closed and convex (refer to Example 1.25 and Example 8.3 in \cite{bauschke2017convex}).
Now, \eqref{mod:fracfg1} is equivalent to
\begin{equation}\label{mod:fracfg2}
\min_{\bx\in\bbRn}\left\{F(\bx)\right\},
\end{equation}
where
\begin{equation}\label{def:funcF}
F(\bx):=\begin{cases}\frac{\tilde{f}(\bx)}{g(\bx)},&\text{if}\ \bx\in\Omega;\\
+\infty,&\text{otherwise}.
\end{cases}
\end{equation}
Thus, $F$ is treated as an extended-real-valued function on $\bbRn$, with
$\dom F=\Omega$. That is, points outside $\Omega$ are excluded from minimization by assigning the value $+\infty$ to $F$. In the subsequent subdifferential and KL analysis, $F$ is always understood in this extended-valued sense.

We next discuss the connection between the reformulated model \eqref{mod:fracfg2} and the fractional model studied in \cite{zhang2022first}. In \cite{zhang2022first}, the fractional expression in their equation (3.1) is used only at feasible points where the nonsmooth term in the numerator  is finite, while the objective is assigned the value $+\infty$ elsewhere. By taking the nonsmooth term in the numerator considered in \cite{zhang2022first} as our $\iota_\Omega$, the smooth term in the numerator as our shifted function $\tilde f=f-Mg$, and the denominator as the present function $g$, the reformulated problem \eqref{mod:fracfg2} can be viewed as a structured specialization of the model considered in \cite{zhang2022first}. We emphasize that the above relationship applies to the reformulated model \eqref{mod:fracfg2}, not directly to the original fractional model \eqref{mod:fracfg1}. We now turn our attention to solving \eqref{mod:fracfg2}.

\subsection{Development of proximal gradient algorithm}
In this subsection, we develop a Proximal Gradient Algorithm (PGA) to solve \eqref{mod:fracfg2}. To this end, we will establish the relationship between the critical point of the model and the proximal characterization (see Lemma \ref{lem:proxcrit}), and then propose the PGA based on this characterization.

We first recall the notions of subdifferentials and critical points \cite{mordukhovich2018variational}. Let $B(\bx;\delta)$ denote the neighborhood of $\bx$ with radius $\delta>0$. The lower limit of function $\psi$ at $\bx$ and the domain of $\psi$ are defined by
$$
\liminf\limits_{\substack{\bm{y}\to\bm{x}\\ \bm{y}\neq\bm{x}}}\psi(\by):=\lim\limits_{\delta\to0^+}\Big(\inf\limits_{\by\in B(\bx;\delta)}\psi(\by)\Big)\ \ \text{and}\ \ \dom\psi:=\{\bx\in\bbRn|\hspace{2pt}\psi(\bx)<+\infty\},
$$
respectively.

\begin{definition}[Subdifferentials and critical point]\label{subdifferentials}
Let $\psi:\bbRn\to\overline{\bbR}$ be a proper lower semicontinuous function.
\begin{itemize}
\item[$(i)$] For each $\bm{x}\in\dom\psi$, the Fr{\'e}chet subdifferential of $\psi$ at $\bm{x}$, denoted by $\hat{\partial}\psi(\bm{x})$, is the set of all vectors $\bm{u}\in\bbRn$ that satisfy $\liminf\limits_{\substack{\bm{y}\to\bm{x}\\ \bm{y}\neq\bm{x}}}\frac{\psi(\bm{y})-\psi(\bm{x})-\langle\bm{u},\bm{y}-\bm{x}\rangle}{\|\bm{y}-\bm{x}\|_2}\geqs0$. When $\bm{x}\notin\dom\psi$, we define $\hat{\partial}\psi(\bm{x})=\varnothing$.
\item[$(ii)$] The limiting-subdifferential, or simply the subdifferential of $\psi$ at $\bm{x}\in\dom\psi$, denoted by $\partial\psi(\bm{x})$, is defined through the following closure process:
\begin{align*}
\partial\psi(\bm{x}):=\{\bm{u}\in\mathbb{R}^n|\hspace{2pt}&\exists\bm{x}^k\to\bm{x},\ \psi(\bm{x}^k)\to\psi(\bm{x})\ \mbox{and}\ \bm{u}^k\in \hat{\partial}\psi(\bm{x}^k)\to\bm{u} \ \mbox{as}\ k\to+\infty\}.
\end{align*}
\end{itemize}
We say that $\bx$ is a critical point of $\psi$ if $\bm{0}_n\in\partial\psi(\bx)$.
\end{definition}

We also recall a known result relating the original fractional function $F$ in \eqref{mod:fracfg2} to its corresponding subtractive form, and then relate the critical point of $F$ to that of the subtractive form.

\begin{lemma}[{\cite[Proposition 3.1]{zhang2022first}}]\label{lem:modeleqglob}
Let the function $F$ be defined by \eqref{def:funcF}, let $\bx^*\in\Omega$, and define $c^*:=F(\bx^*)$. Then, $\bx^*$ is a globally optimal solution of \eqref{mod:fracfg2} if and only if it is a globally optimal solution of
\begin{equation}\label{model:subtraction}
\min_{\bx\in\bbR^n}\ \tilde{f}(\bx)+\iota_{\Omega}(\bx)-c^*g(\bx).
\end{equation}
\end{lemma}

\begin{proposition}\label{prop:critpoint}
Let the function $F$ be defined by \eqref{def:funcF}, let $\bx^*\in\Omega$, and define $c^*:=F(\bx^*)$. Then, $\bx^*$ is a critical point of $F$ if and only if $\bm{0}_n\in\partial\iota_{\Omega}(\bx^*)+\nabla\tilde{f}(\bx^*)-c^*\nabla g(\bx^*)$.
\end{proposition}
\begin{proof}
Let $\bar{\bx}\in\Omega$ be arbitrary. Since $g(\bar{\bx})>0$ and $g$ is continuous,
there exists a neighborhood $U$ of $\bar{\bx}$ such that $g(\bx)>0$ for all
$\bx\in U$. Define $q(\bx):=\frac{\tilde f(\bx)}{g(\bx)}$ for all $\bx\in U$. By the differentiability of $\tilde f$ and $g$ on $\Omega$, $q$ is
differentiable at every $\bx\in U\cap\Omega$. Moreover, the restriction
of $F$ to $U$ coincides with $q+\iota_{\Omega}$. Thus, for
$\bx\in U\cap\Omega$, the definition of the Fr{\'e}chet subdifferential gives
\begin{equation}\label{neq1:propcritpoint}
\hat\partial F(\bx)=\nabla q(\bx)+\hat\partial\iota_{\Omega}(\bx)=\nabla q(\bx)+\partial\iota_{\Omega}(\bx),
\end{equation}
where the second equality follows from the convexity of $\Omega$. Since $\nabla\tilde f$ and $\nabla g$ are Lipschitz continuous on $\Omega$, $\nabla q$ is continuous on $U\cap\Omega$. The definition of the limiting subdifferential and the closedness of the graph of $\partial\iota_{\Omega}$ give $\partial F(\bar{\bx})\subseteq\nabla q(\bar{\bx})+\partial\iota_{\Omega}(\bar{\bx})$. Moreover, by \eqref{neq1:propcritpoint} and $\hat\partial F(\bar{\bx})\subseteq\partial F(\bar{\bx})$, we have $\nabla q(\bar{\bx})+\partial\iota_{\Omega}(\bar{\bx})\subseteq \partial F(\bar{\bx})$. Hence
$\partial F(\bar{\bx})=\nabla q(\bar{\bx})+\partial\iota_{\Omega}(\bar{\bx})$.
By the quotient rule, we have
\[
\nabla q(\bar{\bx})
=
\frac{g(\bar{\bx})\nabla\tilde f(\bar{\bx})
-\tilde f(\bar{\bx})\nabla g(\bar{\bx})}{g(\bar{\bx})^2}.
\]
Since $F(\bar{\bx})=\frac{\tilde f(\bar{\bx})}{g(\bar{\bx})}$, we get
\[
\partial F(\bar{\bx})
=
\partial\iota_{\Omega}(\bar{\bx})
+\frac{1}{g(\bar{\bx})}
\left(\nabla\tilde f(\bar{\bx})-F(\bar{\bx})\nabla g(\bar{\bx})\right).
\]
Furthermore, since $\Omega$ is convex, $\partial\iota_{\Omega}(\bar{\bx})$ is a convex
cone, which together with $g(\bar{\bx})>0$ implies that
\begin{equation}\label{eq:subdiffF}
\partial F(\bar{\bx})=\frac{1}{g(\bar{\bx})}\left(\partial\iota_{\Omega}(\bar{\bx})+\nabla\tilde f(\bar{\bx})-F(\bar{\bx})\nabla g(\bar{\bx})\right).
\end{equation}
Taking $\bar{\bx}=\bx^*$ in \eqref{eq:subdiffF} and using $c^*=F(\bx^*)$, we conclude
that $\bx^*$ is a critical point of $F$ if and only if
\[
\bm{0}_n\in
\partial\iota_{\Omega}(\bx^*)+\nabla\tilde f(\bx^*)-c^*\nabla g(\bx^*).
\]
This completes the proof.
\end{proof}

For the development of the Proximal Gradient Algorithm (PGA), we recall that the proximity operator of a proper function $\psi:\bbR^n\to\overline{\bbR}$ at $\bx\in\bbRn$ is defined by
\begin{equation}\label{def_prox}
\prox_{\psi}(\bx):=\argmin_{\bu\in\bbRn}\left\{\frac{1}{2}\|\bu-\bx\|_2^2+\psi(\bu)\right\}.
\end{equation}
From the definition of the operator $\prox_{\iota_{\Omega}}$, it is straightforward to see that it is exactly the projection operator $P_{\Omega}$ onto the closed convex $\Omega$, that is,
\begin{equation}\label{eq:proxproj}
\prox_{\iota_{\Omega}}(\bx)=P_{\Omega}(\bx)=\argmin_{\bu\in\Omega}\|\bu-\bx\|_2.
\end{equation}

We then recall a sufficient condition for a critical point of function $F$ defined by \eqref{def:funcF} based on the proximal characterization.

\begin{lemma}[{\cite[Proposition 4.1]{zhang2022first}}]\label{lem:proxcrit}
If $\bx^*\in\bbRn$ satisfies
\begin{equation}\label{eq_proxchar}
\bx^*=\prox_{\iota_{\Omega}}(\bx^*-\alpha\nabla\tilde{f}(\bx^*)+\alpha c^*\nabla g(\bx^*))
\end{equation}
for some $\alpha>0$ and $c^*:=F(\bx^*)$, then $\bx^*$ is a critical point of $F$.
\end{lemma}

In light of Lemma \ref{lem:proxcrit}, we present the following proximal gradient algorithm, which, under the present differentiability assumption on $g$, can be viewed as a special case of the proximity-gradient-subgradient algorithm proposed in \cite{zhang2022first}. In this paper, we use $\bbN$ to denote the set of nonnegative integers and $\bbN_+$ to denote the set of positive integers, respectively. We define $\bbR_+:=(0,+\infty)$.\vspace{0.5em}

\begin{breakablealgorithm}\label{algo0}
\caption{\textbf{Proximal Gradient Algorithm (PGA)}}
\begin{algorithmic}
{\bf Initialization:} Preset $\alpha\in\bbR_+$ and choose an initial vector $\bx^0\in\Omega$.\\
For each $k\in\bbN$, generate the sequence $\{\bx^k\}_{k\in\bbN}$ as follows:
\begin{equation}\label{PGAiter}
\begin{cases}
c_k=\frac{\tilde{f}(\bx^k)}{g(\bx^k)}\\
\bx^{k+1}=\prox_{\iota_{\Omega}}(\bx^k-\alpha\nabla\tilde{f}(\bx^k)+\alpha c_k\nabla g(\bx^k))
\end{cases}
\end{equation}
\end{algorithmic}
\end{breakablealgorithm}

\subsection{Convergence analysis of PGA}
In this subsection, we analyze the convergence of PGA. We aim to prove that under certain conditions, the iterative sequence generated by PGA converges to a globally optimal solution of \eqref{mod:fracfg2}, and consequently, to a globally optimal solution of \eqref{mod:fracfg1}. For this purpose, we revisit the concept of KL property \cite{attouch2010proximal}.

\begin{definition}[KL property]\label{def:KLproperty}
Let $\psi:\bbR^n\to\overline{\bbR}$ be a proper lower semicontinuous function. We say that $\psi$ satisfies the KL property at $\hat{\bx}\in \text{dom }\partial\psi$ if there exist $\eta\in(0,+\infty]$, a neighborhood $U$ of $\hat{\bx}$, and a continuous concave function $\varphi:[0,\eta)\to[0,+\infty]$ such that
\begin{itemize}
\item[$(i)$] $\varphi(0)=0$;
\item[$(ii)$] $\varphi$ is continuously differentiable on $(0,\eta)$ with $\varphi'>0$;
\item[$(iii)$]  $\varphi'(\psi(\bx)-\psi(\hat{\bx}))\cdot$dist$(0,\partial\psi(\bx))\geqs 1$ for any $\bx\in U\cap\{\bx\in\bbR^n:\psi(\hat{\bx})<\psi(\bx)<\psi(\hat{\bx})+\eta\}$.
\end{itemize}
\end{definition}
\begin{definition}[KL function]
We call a proper lower semicontinuous function $\psi:\bbR^n\to\overline{\bbR}$ KL function if $\psi$ satisfies the KL property at all points in $\dom\partial\psi$.
\end{definition}

Since the PGA is a special case of the proximity-gradient-subgradient algorithm in \cite{zhang2022first}, it follows from Theorem 4.8 therein that, to establish convergence of the PGA to a critical point of $F$, it suffices to verify that $F$ satisfies the KL property at any point in $\dom F$. To this end, we recall several properties of semi-algebraic functions from \cite{attouch2010proximal} and \cite{coste2002introduction} as Lemma \ref{lem:semialgeproper}, and recall Theorem 3 in \cite{bolte2014proximal} that establishes the relation between semi-algebraic property and KL property as Lemma \ref{lem:semialgeKL}.

\begin{lemma}\label{lem:semialgeproper}
Let $\mathcal{S}\subset\bbRn$ be a semi-algebraic set. The following statements are true:
\begin{itemize}
\item[$(i)$] The finite sum or product of semi-algebraic functions is semi-algebraic.
\item[$(ii)$] If $\psi:\mathcal{S}\to\overline{\bbR}$ is semi-algebraic and $\psi(\bx)\geqs0$ for all $\bx\in\mathcal{S}$, then $\sqrt{\psi}$ is semi-algebraic.
\end{itemize}
\end{lemma}

\begin{lemma}\label{lem:semialgeKL}
Let $\psi:\bbR^n\to\overline{\bbR}$ be a proper lower semicontinuous function. If $\psi$ is semi-algebraic, then it satisfies the KL property at any point of $\dom\psi$.
\end{lemma}

\begin{proposition}\label{prop_Fsemi}
Function $F$ defined by \eqref{def:funcF} satisfies the KL property at any point of $\dom F$.
\end{proposition}
\begin{proof}
According to Lemma \ref{lem:semialgeKL}, to prove the desired result, it suffices to show that $F$ is semi-algebraic. We recall from Assumption 1 that $f$ and $g$ are both semi-algebraic on $\Omega$, where $\Omega$ is a semi-algebraic set. Then we know from item $(i)$ in Lemma \ref{lem:semialgeproper} that $\tilde{f}$ is also semi-algebraic. It remains to verify the semi-algebraicity of $F$ in its extended-valued form. Since $g(\bx)>0$ for all $\bx\in\Omega$, the epigraph of $F$ can be written as
\[
\operatorname{epi} F=
\left\{(\bx,t)\in\bbRn\times\bbR\mid
\bx\in\Omega,\ \tilde f(\bx)\leqs t g(\bx)
\right\}.
\]
Using the fact that $\Omega$, $\tilde f$, and $g$ are all semi-algebraic, the above representation shows that $\operatorname{epi} F$ is semi-algebraic, and hence $F$ is semi-algebraic. This completes the proof.
\end{proof}

We next recall two results from Lemmas 4.6 and 4.7 in \cite{zhang2022first} as the following Lemma \ref{lem_Fmonodec}, and then establish a sufficient condition for a critical point of $F$ to be a global minimizer.

\begin{lemma}\label{lem_Fmonodec}
Let $\{\bx^k\}_{k\in\bbN}$ be generated by PGA, where $\alpha\in\left(0,\frac{1}{L_{\nabla\tilde{f}}}\right)$. Then the following hold:
\begin{itemize}
\item[$(i)$] There exists $a\in\bbR_+$ such that $F(\bx^{k+1})+a\|\bx^{k+1}-\bx^k\|_2^2\leqs F(\bx^k)$ for all $k\in\bbN$.
\item[$(ii)$] There exist $\bq^{k+1}\in\partial F(\bx^{k+1})$ and $b\in\bbR_+$ such that $\|\bq^{k+1}\|_2\leqs b\|\bx^{k+1}-\bx^k\|_2$, for $k\in\bbN$.
\end{itemize}
\end{lemma}

\begin{proposition}\label{prop:globopt}
Let $\bx^*\in\Omega$ be a critical point of $F$. If $f(\bx^*)\leqs0$, then $\bx^*$ is a globally optimal solution of \eqref{mod:fracfg2}, that is, a globally optimal solution of \eqref{mod:fracfg1}.
\end{proposition}
\begin{proof}
Let $\tilde{F}(\bx;\bx^*):=\tilde{f}(\bx)+\iota_{\Omega}(\bx)-c^*g(\bx)$ and $\bx\in\bbRn$. Then it follows from Proposition \ref{prop:critpoint} that $\bx^*$ is a critical point of $\tilde{F}(\hspace{1pt}\cdot\hspace{2pt};\bx^*)$. According to the definition of $\tilde{f}$ in \eqref{def:tildef}, we get that
\begin{align*}
\tilde{F}(\bx;\bx^*)&=f(\bx)-Mg(\bx)+\iota_{\Omega}(\bx)-\frac{f(\bx^*)-Mg(\bx^*)}{g(\bx^*)}\cdot g(\bx)\\
&=f(\bx)+\iota_{\Omega}(\bx)-\frac{f(\bx^*)}{g(\bx^*)}\cdot g(\bx).
\end{align*}
We recall that the functions $f$, $\iota_{\Omega}$, and $g$ are all convex. Since $g(\bx^*)>0$ and $f(\bx^*)\leqs0$, we observe that $\tilde{F}(\bx;\bx^*)$ is also a convex function with respect to $\bx$. The convexity of $\tilde{F}(\hspace{1pt}\cdot\hspace{2pt};\bx^*)$ ensures that the critical point $\bx^*$ of $\tilde{F}(\hspace{1pt}\cdot\hspace{2pt};\bx^*)$ is indeed a global minimizer of $\tilde{F}(\hspace{1pt}\cdot\hspace{2pt};\bx^*)$. Consequently, we conclude from Lemma \ref{lem:modeleqglob} that $\bx^*$ is a globally optimal solution of \eqref{mod:fracfg2}, which, due to the equivalence of \eqref{mod:fracfg1} and \eqref{mod:fracfg2}, is also a globally optimal solution of \eqref{mod:fracfg1}.
\end{proof}

Combining Proposition \ref{prop:globopt} with
Theorem 4.8 in \cite{zhang2022first}, we obtain the following theorem. Specifically, the convergence to a critical point follows from \cite{zhang2022first}, while the
global optimality assertion follows from Lemma
\ref{lem_Fmonodec} and Proposition \ref{prop:globopt} under the
condition $f(\bx^0)\leqs0$.

\begin{theorem}\label{thm:globalmin}
Let $\{\bx^k\}_{k\in\bbN}$ be generated by PGA, where
$\alpha\in\left(0,\frac{1}{L_{\nabla\tilde{f}}}\right)$. Then
$\{\bx^k\}_{k\in\bbN}$ converges to a critical point $\bx^*\in\Omega$ of $F$.
Moreover, if the initial point $\bx^0$ satisfies $f(\bx^0)\leqs0$, then $\bx^*$ is a globally optimal solution of model \eqref{mod:fracfg2}, that is, a globally optimal solution of model \eqref{mod:fracfg1}.
\end{theorem}
\begin{proof}
The convergence of $\{\bx^k\}_{k\in\bbN}$ to a critical point $\bx^*\in\Omega$ of $F$ follows directly from Theorem 4.8 in \cite{zhang2022first}. It remains to prove the global optimality assertion. By item $(i)$ of Lemma \ref{lem_Fmonodec}, we have
$F(\bx^k)\leqs F(\bx^0)$ for all $k\in\bbN$. Letting $k\to\infty$ and using the
continuity of $F$ on $\Omega$, we obtain $F(\bx^*)\leqs F(\bx^0)$. Since $F(\bx)=\frac{\tilde f(\bx)}{g(\bx)}=\frac{f(\bx)}{g(\bx)}-M$, it follows that
$\frac{f(\bx^*)}{g(\bx^*)}\leqs\frac{f(\bx^0)}{g(\bx^0)}\leqs0$, which together with $g(\bx^*)>0$ implies that $f(\bx^*)\leqs0$. Therefore,
Proposition \ref{prop:globopt} shows that $\bx^*$ is a globally optimal solution of
\eqref{mod:fracfg2}, and hence also of \eqref{mod:fracfg1}.
\end{proof}

We next investigate the convergence rate of PGA under the KL property of the objective
function. For this purpose, we recall the following lemma from
\cite{attouch2009convergence}.

\begin{lemma}\label{lem:forconvrate1}
Let $\gamma\in\bbR_+$, $\{a_k\}_{k\in\bbN}\subset\bbR_+$ be a monotonically decreasing sequence such that $\lim_{k\to\infty}a_k=0$. Suppose that there exist $\mu_1\in[0,+\infty)$, $\mu_2\in\bbR_+$ and $K'\in\bbN$ such that
\begin{equation}\label{neq:akreak}
a_k\leqs\mu_1(a_{k-1}-a_k)+\mu_2(a_{k-1}-a_k)^{\gamma},\ \ \mbox{for all}\ k\geqs K'.
\end{equation}
Then the following hold:
\begin{itemize}
\item[$(i)$] If $\gamma\in\left[1,+\infty\right)$, then there exist $\mu'>0$, $\tau\in[0,1)$ and $K\geqs K'$ such that
    $$
    a_k\leqs\mu'\tau^k,\ \ \mbox{for all}\ k\geqs K.
    $$
\item[$(ii)$] If $\gamma\in\left(0,1\right)$, then there exist $\mu'>0$ and $K\geqs K'$ such that
    $$
    a_k\leqs \mu'k^{-\frac{\gamma}{1-\gamma}},\ \ \mbox{for all}\ k\geqs K.
    $$
\end{itemize}
\end{lemma}

In the following theorem, we present a convergence rate result of PGA based on the analysis of convergence rate in \cite{attouch2009convergence}. We remark that in addition to the convergence rate of $\|\bx^k-\bx^*\|_2$, we also provide a convergence rate result of $F(\bx^k)-F(\bx^*)$ here, which is not given in \cite{attouch2009convergence}.

\begin{theorem}
Let $F$ be defined by \eqref{def:funcF}, $\{\bx^k\}_{k\in\bbN}$ be generated by PGA, where $\alpha\in\left(0,\frac{1}{L_{\nabla\tilde{f}}}\right)$. Assume that $\{\bx^k\}_{k\in\bbN}$ converges to a critical point $\bx^*$ of $F$, and that $F$ has the KL property at $\bx^*$ with $\varphi(s):=\mu s^{(1-\theta)}$, $\mu>0$. Then the following estimations hold:
\begin{itemize}
\item[$(i)$] If $\theta=0$, then the sequence $\{\bx^k\}_{k\in\bbN}$ converges in a finite number of steps.
\item[$(ii)$] If $\theta\in\left(0,\frac{1}{2}\right]$, then there exist $\tilde{\mu}_1,\tilde{\mu}_2\in\bbR_+$, $\tau_1,\tau_2\in[0,1)$ and $K\in\bbN$ such that
$$
\|\bx^k-\bx^*\|_2\leqs\tilde{\mu}_1\tau_1^k\ \ \mbox{and}\ \ F(\bx^k)-F(\bx^*)\leqs\tilde{\mu}_2\tau_2^k,\ \ \forall\ k\geqs K.
$$
\item[$(iii)$] If $\theta\in\left(\frac{1}{2},1\right)$, then there exist $\tilde{\mu}_1,\tilde{\mu}_2\in\bbR_+$ and $K\in\bbN$ such that
$$
\|\bx^k-\bx^*\|_2\leqs \tilde{\mu}_{1}k^{-\frac{1-\theta}{2\theta-1}}\ \ \mbox{and}\ \ F(\bx^k)-F(\bx^*)\leqs\tilde{\mu}_2k^{-\frac{1}{2\theta-1}},\ \ \forall\ k\geqs K.
$$
\end{itemize}
\end{theorem}
\begin{proof}
We first prove item $(i)$. If $\{F(\bx^k)\}_{k\in\bbN}$ is stationary, then so is $\{\bx^k\}_{k\in\bbN}$ in view of item $(i)$ of Lemma \ref{lem_Fmonodec}.  Assume that $\{F(\bx^k)\}_{k\in\bbN}$ is not stationary. Since $\theta=0$, the KL inequality in item $(iii)$ of Definition \ref{def:KLproperty} yields that for any sufficiently large $k$, $\mu\cdot\dist({\bm0},\partial F(\bx^k))\geqs1$, contradicts the fact that $\bx^*$ is a critical point of $F$.

We then prove item $(ii)$ and $(iii)$. Let $s_k:=\sum_{i=k}^{\infty}\|\bx^{i+1}-\bx^{i}\|_2$, $k\in\bbN$ and $l_i:=F(\bx^i)-F(\bx^*)$, $i\in\bbN$. By using the triangle inequality, it is easy to see that $s_k\geqs\|\bx^{k}-\bx^*\|_2$ for $k\in\bbN$. To analyze the convergence rate of $\|\bx^k-\bx^*\|$, it suffices to analyze the convergence rate of $s_k$. For $\theta\in(0,1)$, the KL inequality yields that there exists $K_1\in\bbN$ such that $\mu(1-\theta)l_i^{-\theta}\cdot\dist({\bm0},\partial F(\bx^i))\geqs 1$, that is, $\mu(1-\theta)\cdot\dist({\bm0},\partial F(\bx^i))\geqs l_i^{\theta}$, for all $i\geqs K_1$. This together with item $(ii)$ of Lemma \ref{lem_Fmonodec} implies that there exists $b>0$ such that
\begin{equation}\label{neq:rexgeqlk}
\mu b(1-\theta)\|\bx^{i}-\bx^{i-1}\|_2\geqs l_i^{\theta}.
\end{equation}
According to item $(i)$ of Lemma \ref{lem_Fmonodec}, there exists $a>0$ such that
\begin{equation}\label{neq:rexisqreli}
a\|\bx^{i+1}-\bx^{i}\|_2^2\leqs l_{i}-l_{i+1}.
\end{equation}
By the definition of $\varphi$, we have $\varphi'(s)=\mu(1-\theta)s^{-\theta}>0$ for $s\in\bbR_+$. Then the concavity of $\varphi$ yields that $a\varphi'(l_i)\|\bx^{i+1}-\bx^{i}\|_2^2\leqs\varphi'(l_i)(l_{i}-l_{i+1})\leqs\varphi(l_i)-\varphi(l_{i+1})$, that is,
\begin{equation}\label{neq:rexleqphilk}
\mu a(1-\theta)\|\bx^{i+1}-\bx^{i}\|_2^2\leqs l_i^{\theta}\big(\varphi(l_i)-\varphi(l_{i+1})\big).
\end{equation}
Combining \eqref{neq:rexgeqlk} and \eqref{neq:rexleqphilk} gives that $\varphi(l_i)-\varphi(l_{i+1})\geqs\frac{a}{b}\cdot\frac{\|\bx^{i+1}-\bx^{i}\|_2^2}{\|\bx^{i}-\bx^{i-1}\|_2}$, that is,
$$
\left[\frac{b}{a}\big(\varphi(l_i)-\varphi(l_{i+1})\big)\right]^{\frac{1}{2}}\|\bx^{i}-\bx^{i-1}\|_2^{\frac{1}{2}}\geqs\|\bx^{i+1}-\bx^{i}\|_2.
$$
Then it follows from the arithmetic mean-geometric mean inequality that
\begin{equation}\label{neq:rephilkrex}
\frac{b}{a}\big(\varphi(l_i)-\varphi(l_{i+1})\big)+\|\bx^{i}-\bx^{i-1}\|_2\geqs2\|\bx^{i+1}-\bx^{i}\|_2.
\end{equation}
Summing both sides of inequality \eqref{neq:rephilkrex} for $i=k,k+1,\ldots,K'$, where $k\geqs K_1$, yields
$$
\frac{b}{a}\big(\varphi(l_k)-\varphi(l_{K'+1})\big)+\|\bx^{k}-\bx^{k-1}\|_2\geqs\sum_{i=k}^{K'}\|\bx^{i+1}-\bx^{i}\|_2+\|\bx^{K'+1}-\bx^{K'}\|_2.
$$
In view of the fact that $\varphi(l_{K'+1})\geqs0$, $\sum_{i=k}^{K'}\|\bx^{i+1}-\bx^{i}\|_2\leqs\|\bx^{k}-\bx^{k-1}\|_2+\frac{b}{a}\varphi(l_k)$. Letting $K'\to\infty$, we obtain
\begin{equation}\label{neq:Skbaphilk}
s_k\leqs(s_{k-1}-s_k)+\frac{b}{a}\varphi(l_k),\ \ \mbox{for all}\ k\geqs K_1.
\end{equation}
It is easy to see from Theorem \ref{thm:globalmin} that $\lim\limits_{k\to\infty}s_{k-1}-s_k=\lim\limits_{k\to\infty}\|\bx^{k}-\bx^{k-1}\|_2=0$. In addition, we have $\lim_{k\to\infty}l_k=0$. Then it follows from \eqref{neq:Skbaphilk} that $\lim_{k\to\infty}s_k=0$.

The definition of $\varphi$ together with \eqref{neq:rexgeqlk} yields that
\begin{equation}\label{neq:philkreSk}
\varphi(l_k)=\mu l_k^{1-\theta}\leqs\mu[\mu b(1-\theta)]^{\frac{1-\theta}{\theta}}(s_{k-1}-s_k)^{\frac{1-\theta}{\theta}}.
\end{equation}
Combining \eqref{neq:Skbaphilk} and \eqref{neq:philkreSk} yields $s_k\leqs(s_{k-1}-s_k)+\frac{b}{a}\mu[\mu b(1-\theta)]^{\frac{1-\theta}{\theta}}(s_{k-1}-s_k)^{\frac{1-\theta}{\theta}}$. Then the desired convergence rate results of $\|\bx^k-\bx^*\|$ in item $(ii)$ and item $(iii)$ follows from Lemma \ref{lem:forconvrate1} together with the fact $s_k\geqs\|\bx^{k}-\bx^*\|_2$.

Finally, we consider the convergence rate results of $F(\bx^k)-F(\bx^*)$. Combining \eqref{neq:rexgeqlk} and \eqref{neq:rexisqreli} yields $l_i\leqs\left[a^{-\frac{1}{2}}\mu b(1-\theta)\right]^{\frac{1}{\theta}}(l_{i-1}-l_i)^{\frac{1}{2\theta}}$ for all $i\geqs K_1$. Then the desired convergence rate results of $F(\bx^k)-F(\bx^*)$ in item $(ii)$ and item $(iii)$ follows from Lemma \ref{lem:forconvrate1} with $a_k=l_k$, $\mu_1=0$, $\mu_2=\left[a^{-\frac{1}{2}}\mu b(1-\theta)\right]^{\frac{1}{\theta}}$ and $\gamma=\frac{1}{2\theta}$.
\end{proof}

Before closing this section, we give an equivalent form of iteration \eqref{PGAiter} in PGA by using $f$ instead of $\tilde{f}$. According to the definition of $\tilde{f}$, it is easy to see that \eqref{PGAiter} is equivalent to
\begin{equation}\label{PGAiter2}
\bx^{k+1}=\prox_{\iota_{\Omega}}\left(\bx^k-\alpha\nabla f(\bx^k)+\alpha\frac{f(\bx^k)}{g(\bx^k)}\nabla g(\bx^k)\right).
\end{equation}

\section{Numerical examples}\label{sec:numerexam}
In the preceding section, we showed that PGA for solving a class of nonconvex fractional optimization problems can converge to a globally optimal solution under certain condition. In this section, we will apply this notable result to the Sharpe Ratio (SR) maximization problem, which is a very important problem in the finance industry due to the SR being a crucial metric for evaluating the risk-adjusted returns. To further assess the effectiveness of PGA, we will then introduce two simple fractional optimization examples that possess closed-form solutions. The first example is a special case of the SR model, whereas the second has an unbounded feasible set and multiple globally optimal solutions while satisfying Assumptions 1 and 2.

\subsection{Sharpe ratio maximization model}
We begin with outlining the SR maximization model. Consider a sample asset return matrix $\bR\in\bbR^{T\times N}$, which represents $T$ trading times and $N$ risky assets. For a given portfolio $\bw\in\bbRN$, the SR is defined by
\begin{equation}\label{eq:SRdef1}
S(\bw):=\frac{\frac{1}{T}\bm{1}_T^{\top}\bR\bw}{\sqrt{\frac{1}{T-1}\|\bR\bw-(\frac{1}{T}\bm{1}_T^{\top}\bR\bw)\bm{1}_T\|_2^2+\hat{\epsilon}\|\bw\|_2^2}},
\end{equation}
where $\hat{\epsilon}\|\bw\|_2^2$ serves as a regularization term for ensuring a positive definite quadratic form. The parameter $\hat{\epsilon}$ may be set to an arbitrarily small positive value, such that its impact on the denominator is negligible. For the sake of brevity, define $\bp:=\frac{1}{T}\bR^{\top}\bm{1}_T$ and $\bQ:=\frac{1}{\sqrt{T-1}}\left(\bR-\frac{1}{T}\bm{1}_{T\times T}\bR\right)$. With these definitions, \eqref{eq:SRdef1} simplifies to
\begin{equation*}
S(\bw)=\frac{\bp^\top\bw}{\sqrt{\bw^\top(\bQ^\top\bQ+\hat{\epsilon}\bI)\bw}}.
\end{equation*}

The task of maximizing $S$ subject to the the long-only and self-financing constraints is formulated as
\begin{equation}\label{model:SR1}
\max_{\bw\in\Omega}\ S(\bw),
\end{equation}
where
\begin{equation}\label{def:SetOmega}
\Omega:=\left\{\bw\in\bbRN|\hspace{2pt}\bw\geqs{\bm0}_N\ \mbox{and}\ \bw^\top{\bm1}_N=1\right\}.
\end{equation}
It is evident that the set $\Omega$, as specified by \eqref{def:SetOmega}, is closed, convex, and semi-algebraic. To solve \eqref{model:SR1}, we recast it into the form of \eqref{mod:fracfg1}. This is achieved by setting
\begin{equation}\label{def:funcfg}
f(\bw):=-\bp^\top\bw\ \ \mbox{and}\ \ g(\bw):=\sqrt{\bw^\top(\bQ^\top\bQ+\hat{\epsilon}\bI)\bw}.
\end{equation}
Consequently, \eqref{model:SR1} can be reformulated as
\begin{equation}\label{model:SR2}
\min_{\bw\in\Omega}\left\{\frac{f(\bw)}{g(\bw)}\right\}.
\end{equation}
The numerator $f(\bw)$ in \eqref{model:SR2} may take either a positive or negative value, which can infringe upon the nonnegative constraint on the numerator as mandated in \cite{zhang2022first}. Nonetheless, as delineated in Assumptions 1-2 in Section \ref{sec:PGAforFracOpt}, our method abstains from enforcing such a nonnegative constraint on the numerator, thereby offering increased flexibility.

Before employing PGA to solve \eqref{model:SR2}, we must first confirm that the functions $f$ and $g$ comply with Assumptions 1 and 2 as outlined in Section \ref{sec:PGAforFracOpt}. Note that the set $\Omega$ is bounded. It is clear that Assumption 2 is met. We now turn our attention to verifying Assumption 1. For any $\bw\in\Omega$, it is readily apparent that $\frac{1}{\sqrt{N}}\leqs\|\bw\|_2\leqs\|\bw\|_1=1$. Consequently, we have $g(\bw)\geqs\sqrt{\frac{\hat{\epsilon}}{N}}>0$. From the definition of $f$ as provided in \eqref{def:funcfg}, it can be easily seen that $f$ is convex, semi-algebraic and differentiable with a Lipschitz continuous gradient on $\Omega$. As for the function $g$, we begin by reformulating it into a more succinct expression. Observe that the matrix $\bQ_{\hat{\epsilon}}:=\bQ^\top\bQ+\hat{\epsilon}\bI$, used in the definition of $g$ in \eqref{def:funcfg}, is symmetric positive definite. It possesses a spectral decomposition $\bQ_{\epsilon}=\bU\Lambda\bU^\top$, where $\bU$ is an orthogonal matrix composed of the eigenvectors of $\bQ_{\epsilon}$, $\Lambda$ is a diagonal matrix with its diagonal entries equivalent to the eigenvalues of $\bQ_{\epsilon}$. By setting $\tilde{\bQ}:=\Lambda^\frac{1}{2}\cdot\bU^\top$, we obtain $\bQ_{\epsilon}=\tilde{\bQ}^\top\tilde{\bQ}$, allowing us to express $g$ as $g(\bw)=\|\tilde{\bQ}\bw\|_2$. Additionally, it is straightforward to determine that $\|\tilde{\bQ}\|_2=\sqrt{\lambda_1}$, where $\lambda_1>0$ represents the largest eigenvalue of $\bQ_{\hat{\epsilon}}$. We proceed to establish in the following proposition that $g$ satisfies the properties in Assumption 1.

\begin{proposition}\label{prop:fung}
Let $g:\bbRN\to\bbR$ be defined by \eqref{def:funcfg}. Then the following statements hold:
\begin{itemize}
\item[$(i)$] $\nabla g$ is $2\lambda_1\sqrt{\frac{N}{\hat{\epsilon}}}$-Lipschitz continuous on $\Omega$;
\item[$(ii)$] $g$ is a semi-algebraic function on $\Omega$.
\end{itemize}
\end{proposition}
\begin{proof}
We first prove item $(i)$. Note that $\nabla g(\bw)=\frac{\tbQ^{\top}\tbQ\bw}{\|\tbQ\bw\|_2}$ for $\bw\in\Omega$. Additionally, for all $\bx,\by\in\Omega$,
$$
\left\|\frac{\tbQ\by}{\|\tbQ\bx\|_2}-\frac{\tbQ\by}{\|\tbQ\by\|_2}\right\|_2=\frac{\big|\|\tbQ\by\|_2-\|\tbQ\bx\|_2\big|\cdot\|\tbQ\by\|_2}{\|\tbQ\bx\|_2\|\tbQ\by\|_2}\leqs\frac{\|\tbQ\bx-\tbQ\by\|_2}{\|\tbQ\bx\|_2}.
$$
This, along with the fact that $\|\tilde{\bQ}\bx\|_2=g(\bx)\geqs\sqrt{\frac{\hat{\epsilon}}{N}}$, gives
\begin{align*}
\|\nabla g(\bx)-\nabla g(\by)\|_2&=\left\|\frac{\tbQ^{\top}\tbQ\bx}{\|\tbQ\bx\|_2}-\frac{\tbQ^{\top}\tbQ\by}{\|\tbQ\by\|_2}\right\|_2\\
&\leqs\left\|\tbQ^\top\right\|_2\left\|\frac{\tbQ\bx-\tbQ\by}{\|\tbQ\bx\|_2}+\frac{\tbQ\by}{\|\tbQ\bx\|_2}-\frac{\tbQ\by}{\|\tbQ\by\|_2}\right\|_2\\
&\leqs2\|\tbQ\|_2\frac{\|\tbQ\bx-\tbQ\by\|_2}{\|\tbQ\bx\|_2}\leqs2\|\tbQ\|_2^2\sqrt{\frac{N}{\hat{\epsilon}}}\|\bx-\by\|_2,
\end{align*}
which implies item $(i)$.

We next prove item $(ii)$. It is evident that $g^2$ is semi-algebraic on $\Omega$ and $g^2(\bx)\geqs0$ for all $\bx\in\Omega$. Then, it follows from item $(ii)$ in Lemma \ref{lem:semialgeproper} that $g$ is semi-algebraic, which completes the proof.
\end{proof}

Note that for all $\bx\in\Omega$, we have $|f(\bx)|\leqs\|\bp\|_2\|\bw\|_2\leqs\|\bp\|_2$. By setting $M=-\|\bp\|_2\sqrt{\frac{N}{\hat{\epsilon}}}$, it follows that $f(\bx)/g(\bx)\geqs M$. Consequently, we have $L_{\nabla\tilde{f}}=\|\bp\|_2\sqrt{\frac{N}{\hat{\epsilon}}}L_{\nabla g}=2\lambda_1\|\bp\|_2\cdot\frac{N}{\hat{\epsilon}}$, where $\tilde{f}$ is defined in \eqref{def:tildef}. The convergence of PGA for solving \eqref{model:SR1} is then an immediate result of Theorem \ref{thm:globalmin}.

\begin{theorem}\label{thm:globalminSRM}
Let $\Omega$ be defined by \eqref{def:SetOmega}, $f$ and $g$ be defined by \eqref{def:funcfg}. Let $\bw^0\in\Omega$ be a given initial vector, and $\{\bw^k\}_{k\in\bbN}$ be a sequence generated by the iteration
\begin{equation}\label{PGAiterforSRM}
\bw^{k+1}=\prox_{\iota_{\Omega}}\left(\bw^k-\alpha\nabla f(\bw^k)+\alpha\frac{f(\bw^k)}{g(\bw^k)}\nabla g(\bw^k)\right),
\end{equation}
where $\alpha\in\left(0,\frac{\hat{\epsilon}}{2N\lambda_1\|\bp\|_2}\right)$. Then the sequence $\{\bw^k\}_{k\in\bbN}$ converges to a critical point $\bw^*\in\Omega$ of function $(f+\iota_{\Omega})/g$. Furthermore, if $\bp^\top\bw^0\geqs0$, then $\bw^*$ is a globally optimal solution of \eqref{model:SR2}, which is equivalent to being a globally optimal solution of  \eqref{model:SR1}.
\end{theorem}

To implement the PGA, the explicit form of $\prox_{\iota_{\Omega}}$ is required. Recall from \eqref{eq:proxproj} that $\prox_{\iota_{\Omega}}$ is precisely the projection operator onto the simplex $\Omega$. The explicit form of this operator is presented in the following Algorithm \ref{algo1}, and further details about the simplex projection algorithm can be found in \cite{simplexproject}.

\vspace{0.3em}
\begin{breakablealgorithm}\label{algo1}
\caption{\textbf{Euclidean Projection onto $\Omega$}}
\begin{algorithmic}
{\bf Input}: A given vector $\bx\in\bbRN$. \\
1. Sort $\bx$ into $\hat{\bx}$ in descending order (i.e. $\hat{x}_{1}\geqs\hat{x}_{2}\geqs\ldots\geqs\hat{x}_{N}$).\\
2. Find $j'=\max\left\{j\in\bbN_{N}: \hat{x}_{j}-\frac{1}{j}\left(\sum_{i=1}^{j}\hat{x}_{i}-1\right)>0\right\}$. \\
3. Compute $\theta=\frac{1}{j'}\left(\sum_{i=1}^{j'} \hat{x}_{i}-1\right)$. \\
{\bf Output}: The projection $\prox_{\iota_{\Omega}}(\bx)=\max\left(\bx-\theta,{\bm0}_N\right)$.
\end{algorithmic}
\end{breakablealgorithm}

The proposed method for Sharpe Ratio Maximization (SRM), abbreviated as SRM-PGA, is summarized in the following Algorithm \ref{algo2}.

\vspace{0.5em}
\begin{breakablealgorithm}\label{algo2}
\caption{\textbf{SRM-PGA}}
\begin{algorithmic}
\noindent{\bf Input:} The sample asset return matrix $\bR\in\bbR^{T\times N}$ and a positive parameter $\hat{\epsilon}$.\\
{\bf Preparation:} Define $\bp=\frac{1}{T}\bR^{\top}\bm{1}_T$, and $\bQ=\frac{1}{\sqrt{T-1}}\left(\bR-\frac{1}{T}\bm{1}_{T\times T}\bR\right)$. Compute the largest eigenvalue  $\lambda_1$ of $\bQ^\top\bQ+\hat{\epsilon}\bI$. Set $\alpha=(0.99\hat{\epsilon})/(2N\lambda_1\|\bp\|_2)$.\\
{\bf Initialization:} Initialize $\bw^0=\frac{1}{N}{\bm1}_{N}$, set tolerance $tol=10^{-5}$, and set the maximum number of iterations $MaxIter=10^{5}$.\\
{\bf Repeat}\\
1. Update $\bw^{k+1}=\prox_{\iota_{\Omega}}\left(\bw^k-\alpha\nabla f(\bw^k)+\alpha\frac{f(\bw^k)}{g(\bw^k)}\nabla g(\bw^k)\right)$.\\
2. Increment the iteration counter $k=k+1$.\\
{\bf Until}: The condition $\frac{\|\bw^{k}-\bw^{k-1}\|_2}{\|\bw^{k-1}\|_2}\leqs tol$ or $k>MaxIter$ is met.\\
3. Set $\bw^*=\bw^{k}$.\\
{\bf Output}: The optimized portfolio $\bw^*$.
\end{algorithmic}
\end{breakablealgorithm}

\subsection{A simple example of the SR model with analytical optimal solution}\label{subsec:sim1}
In this subsection, we present a simple example to verify that PGA converges to the optimal solution of \eqref{mod:fracfg1}. Consider the set
\begin{equation}\label{def:Omega2}
\Omega_2:=\left\{\bx\in\bbR^2|\hspace{2pt}x_1\geqs0,\ x_2\geqs0\ \mbox{and}\ x_1+x_2=1\right\},
\end{equation}
and the function
$$
f(\bx):=\bp^{\top}\bx\ \mbox{and}\ g(\bx):=\|\bx\|_2,\ \bx\in\bbR^2
$$
in \eqref{mod:fracfg1}. The model \eqref{mod:fracfg1} then simplifies to
\begin{equation}\label{mod:toy}
\min_{\bx\in\Omega_2}\left\{\frac{\bp^{\top}\bx}{\|\bx\|_2}\right\},
\end{equation}
which is a special case of the SR model in \eqref{model:SR2} with $\bQ={\bm0}$ and $\hat{\epsilon}=1$ in the definition \eqref{def:funcfg} of $g$. For this example, a computable lower bound can be obtained directly. Since $\Omega_2$ is the two-dimensional simplex, we have $\|\bx\|_2\geqs\frac{\sqrt{2}}{2}$ and $|\bp^\top\bx|\leqs\|\bp\|_2$ for all
$\bx\in\Omega_2$. By taking $M=-\sqrt{2}\|\bp\|_2$, we see that $\frac{\bp^\top\bx}{\|\bx\|_2}\geqs M$ on $\Omega_2$. Note that $\nabla g(\bx)=\frac{\bx}{\|\bx\|_2}$ is Lipschitz continuous on $\Omega_2$ with constant $2\sqrt{2}$. The Lipschitz constant of $\nabla\tilde f(\bx)=\bp-M\cdot\frac{\bx}{\|\bx\|_2}$ is given by $L_{\nabla\tilde f}=4\|\bp\|_2$. The convergence of PGA for solving \eqref{mod:toy} is then an immediate result of Theorem \ref{thm:globalmin}.

The closed-form of $\prox_{\iota_{\Omega_2}}$ is detailed in Algorithm \ref{algo1}. The iterative scheme of PGA for solving the simplified model \eqref{mod:toy} (abbreviated as Sim1-PGA) is then described by
$$
\bx^{k+1}=\prox_{\iota_{\Omega_2}}\left(\bx^k-\alpha\bp+\alpha\frac{\bp^{\top}\bx^k}{\|\bx^k\|_2^2}\bx^k\right),
$$
where $\alpha\in\left(0,\frac{1}{4\|\bp\|_2}\right)$.

To better evaluate the performance of Sim1-PGA, we also derive the closed-form of the globally optimal solution for \eqref{mod:toy} in the following proposition.

\begin{proposition}\label{prop:toyglobal}
Suppose that $\bp\in\bbR^2$ is a vector such that $p_1\neq0$, $p_2\neq0$, $p_1+p_2\neq0$, and $p_1\neq p_2$. Define $\hat{\Omega}_2:=\{\bx\in\bR^2|\hspace{2pt}x_1<0\ \mbox{and}\ x_2<0\}$ and
$$
\bx^*:=\begin{cases}
\left(\frac{p_1}{p_1+p_2},\frac{p_2}{p_1+p_2}\right)^\top,&\mbox{if}\ p\in\hat{\Omega}_2;\\
(0,1)^\top,&\mbox{if}\ p\notin\hat{\Omega}_2\ \mbox{and}\ p_1>p_2;\\
(1,0)^\top,&\mbox{if}\ p\notin\hat{\Omega}_2\ \mbox{and}\ p_1<p_2.
\end{cases}
$$
Then $\bx^*$ is the globally optimal solution of \eqref{mod:toy}.
\end{proposition}
\begin{proof}
If follows from the definition of $\Omega_2$ in \eqref{def:Omega2} that for $\bx\in\Omega_2$, we have $\frac{\bp^{\top}\bx}{\|\bx\|_2}=\frac{(p_1-p_2)x_1+p_2}{\sqrt{2x_1^2-2x_1+1}}$, where $x_1\in[0,1]$. Define $\phi(t):=\frac{(p_1-p_2)t+p_2}{\sqrt{2t^2-2t+1}}$ for $t\in[0,1]$. It is easily verified that $\phi'(t^*)=0$ if and only if $t^*=\frac{p_1}{p_1+p_2}$. Since $\phi$ is differentiable on $(0,1)$ and continuous on $[0,1]$, $\phi$ must attain its minimum value at $0$, $1$, or $t^*$ when $t^*\in[0,1]$. We compute that
$\phi(0)=p_2$, $\phi(1)=p_1$, and $\phi(t^*)=\frac{|p_1+p_2|}{p_1+p_2}\sqrt{p_1^2+p_2^2}$. If $\bp\in\hat{\Omega}_2$, that is, $p_1<0$ and $p_2<0$, then $t^*\in(0,1)$ and $\phi(t^*)=-\sqrt{p_1^2+p_2^2}<\min\{\phi(0),\phi(1)\}$. Hence, the globally optimal solution of \eqref{mod:toy} is $\left(\frac{p_1}{p_1+p_2},\frac{p_2}{p_1+p_2}\right)^\top$.

We then consider the case where $\bp\notin\hat{\Omega}_2$. In this case, if $p_1>0$ and $p_2<0$, or $p_1<0$ and $p_2>0$, then $t^*\notin[0,1]$. If $p_1>0$ and $p_2>0$, then $t^*\in[0,1]$, but $\phi(t^*)=\sqrt{p_1^2+p_2^2}>\max\{\phi(0),\phi(1)\}$. Thus, $\phi$ attains its minimum value at $0$ or $1$. It is now obvious that the globally optimal solution of \eqref{mod:toy} is $(0,1)^\top$ if $p_2<p_1$, or $(1,0)^\top$ if $p_1<p_2$, which completes the proof.
\end{proof}

Two experiments showing that Sim1-PGA converges to the globally optimal solution of \eqref{mod:toy} with two different choices of $\bp$ will be presented in Section \ref{exp:toyexample}.

\subsection{An example with unbounded constraint set and multiple globally optimal solutions}
\label{sec:exampunbound}
We present another example with an unbounded feasible set and multiple
globally optimal solutions. Consider
\begin{gather}
\label{def:Omega2prime}\Omega_2':=\left\{\bx\in\bbR^2\big|\hspace{2pt}|x_2|\leqs a_0\right\},\ \ a_0\in\bbR_+,\\
\label{def:matAB}\bA:=\left[\begin{array}{ll}
a_1 & 0 \\
0 & a_2
\end{array}\right],\ \ \bB:=\left[\begin{array}{ll}
a_4 & 0 \\
0 & a_5
\end{array}\right],\ \ a_i\in\bbR_+,\ i=1,2,4,5,\\
\label{def:fandgquadr}f(\bx):=\bx^\top \bA \bx+a_3 \  \mbox{and} \ g(\bx):=\bx^\top \bB \bx+a_6,\ \ a_3,a_6\in\bbR_+,
\end{gather}
where $a_1,a_2,\ldots a_6$ are such that
\begin{equation}\label{eqn:toy3para}
a_1a_5>a_2a_4\ \mbox{and}\ a_3a_5=a_2a_6.
\end{equation}
Then model \eqref{mod:fracfg1} becomes
\begin{equation}\label{mod:toy3}
\min_{\bx\in\Omega_2'}\left\{\frac{\bx^\top \bA \bx+a_3}{\bx^\top \bB \bx+a_6}\right\}.
\end{equation}
The set $\Omega_2'$ is closed, convex, semi-algebraic, and unbounded, and
$f(\bx)\geqs a_3>0$ for all $\bx\in\Omega_2'$. Moreover, $f$ and $g$ are
convex and semi-algebraic, $g(\bx)\geqs a_6>0$ on $\Omega_2'$, and
$\nabla f(\bx)=2\bA\bx$ and $\nabla g(\bx)=2\bB\bx$ are Lipschitz
continuous on $\Omega_2'$. Thus, Assumption 1 in Section
\ref{sec:PGAforFracOpt} is satisfied. We now verify Assumption 2. For any given $d\in\left\{\frac{f(\bx)}{g(\bx)}\Big|\hspace{2pt}\bx\in\Omega_2'\right\}$, consider the level set $\left\{\bx\in\Omega_2'\Big|\hspace{2pt} \frac{f(\bx)}{g(\bx)}\leqs d \right\}$. With the equality $a_3a_5=a_2a_6$, it follows that
\begin{equation}\label{mod:toy3obj}
\frac{f(\bx)}{g(\bx)}=\frac{\left(a_1-\frac{a_2a_4}{a_5}\right)x_1^2}{a_4x_1^2+a_5x_2^2+a_6}+\frac{a_2}{a_5},\ \ \bx\in\bbR^2,
\end{equation}
which implies that $d\geqs\frac{a_2}{a_5}$. Of course, we also have $\frac{f(\bx)}{g(\bx)}\geqs\frac{a_2}{a_5}$ on $\Omega_2'$, and we can take $M=\frac{a_2}{a_5}$. We next prove that $d<\frac{a_1}{a_4}$. From \eqref{eqn:toy3para}, we have
\begin{equation}\label{neq:a1a6ga3a4}
a_1a_6=\frac{a_1a_3a_5}{a_2}>\frac{a_2a_3a_4}{a_2}=a_3a_4.
\end{equation}
The inequalities \eqref{neq:a1a6ga3a4} and $a_1a_5>a_2a_4$ imply that
$$
a_1a_4x_1^2+a_2a_4x_2^2+a_3a_4<a_1a_4x_1^2+a_1a_5x_2^2+a_1a_6,
$$
that is, $\frac{f(\bx)}{g(\bx)}<\frac{a_1}{a_4}$ for all $\bx\in\bbR^2$. Thus, $d\in\left[\frac{a_2}{a_5},\frac{a_1}{a_4}\right)$. Let $d':=\frac{a_5}{a_1a_5-a_2a_4}\left(d-\frac{a_2}{a_5}\right)$. It is easy to verify that $1-a_4d'>0$. According to \eqref{mod:toy3obj}, $\frac{f(\bx)}{g(\bx)}\leqs d$ implies that $\frac{x_1^2}{a_4x_1^2+a_5x_2^2+a_6}\leqs d'$, that is, $(1-a_4d')x_1^2\leqs a_5d'x_2^2+a_6d'$. This implies the boundedness of the level set $\left\{\bx\in\Omega_2'\Big|\hspace{2pt} \frac{f(\bx)}{g(\bx)}\leqs d \right\}$, since $|x_2|\leqs a_0$ for $\bx\in\Omega_2'$. Therefore, Assumption 2 in Section \ref{sec:PGAforFracOpt} is also satisfied.

From the choice of $M$ and the equality $a_3a_5=a_2a_6$, we obtain
\[
\tilde f(\bx)=f(\bx)-Mg(\bx)=\left(a_1-\frac{a_2a_4}{a_5}\right)x_1^2.
\]
Hence, $\nabla\tilde f(\bx)=2\left(a_1-\frac{a_2a_4}{a_5}\right)(x_1,0)^\top$,
and the Lipschitz constant of $\nabla\tilde f$ is given by $L_{\nabla\tilde f}=2\left(a_1-\frac{a_2a_4}{a_5}\right)$.

It is easy to verify that
$$
\prox_{\iota_{\Omega_2'}}(\bx)=
\left(x_1,\ \sign({x}_2)\cdot\min\{|{x}_2|,a_0\}\right)^\top,\ \ \mbox{for}\ \bx\in\bbR^2.
$$
The iterative scheme of PGA for solving the simple model \eqref{mod:toy3} (abbreviated as Sim2-PGA) can then be given by
$$
\bx^{k+1}=\prox_{\iota_{\Omega_2'}}\left(\bx^k-2\alpha\bA \bx^k+2\alpha\frac{(\bx^{k})^{\top}\bA \bx^k+a_3}{(\bx^{k})^{\top}\bB \bx^k+a_6}\bB\bx^k\right),
$$
where $\alpha\in\left(0,\frac{a_5}{2(a_1a_5-a_2a_4)}\right)$.

According to Theorem \ref{thm:globalmin}, the sequence generated by Sim2-PGA converges to a critical point $\bx^*\in\Omega_2'$ of $F$. The following proposition characterizes this limiting point and establishes its global optimality for \eqref{mod:toy3}.

\begin{proposition}\label{prop:toy3global}
Let $\Omega_2'$ be defined by \eqref{def:Omega2prime}, and let $f$ and $g$ be defined by \eqref{def:fandgquadr}. Let $F$ be defined by \eqref{def:funcF} for this example,
with $\Omega=\Omega_2'$ and $M=\frac{a_2}{a_5}$. If $\bx^*\in\Omega_2'$ is a critical point of $F$, then $\bx^*=(0,x_2^*)^\top$ with $|x_2^*|\leqs a_0$, and $\bx^*$ is a globally optimal solution of \eqref{mod:toy3}. In addition, the set of all globally optimal solutions of \eqref{mod:toy3} is
\[
\left\{\bx\in\bbR^2\big|\hspace{2pt}x_1=0,\ |x_2|\leqs a_0\right\}.
\]
\end{proposition}
\begin{proof}
For any $\bv=(v_1,v_2)^\top\in\partial\iota_{\Omega_2'}(\bx^*)$, since $\iota_{\Omega_2'}$ is a convex function, we have
\begin{equation}\label{neq:proptoyexam2}
\langle\bv,\bx-\bx^*\rangle\leqs0,\ \ \forall \bx\in\Omega_2'.
\end{equation}
Note that $(x_1^*+1,x_2^*)^\top$ and
$(x_1^*-1,x_2^*)^\top$ are both in $\Omega_2'$. Substituting these two points into $\bx$ in \eqref{neq:proptoyexam2}, we obtain that $v_1\leqs0$ and $-v_1\leqs0$, respectively, and hence $v_1=0$. Let $\bv^*:=F(\bx^*)\nabla g(\bx^*)-\nabla\tilde f(\bx^*)$. Recalling that
$\nabla\tilde f(\bx)=2\left(a_1-\frac{a_2a_4}{a_5}\right)(x_1,0)^\top$
and $\nabla g(\bx)=2\bB\bx$, we have
$v_1^*=2x_1^*\left(a_4F(\bx^*)-a_1+\frac{a_2a_4}{a_5}\right)$. In addition, since $\bx^*$ is a critical point of $F$, it follows from Proposition
\ref{prop:critpoint} that $\bv^*\in\partial\iota_{\Omega_2'}(\bx^*)$, which implies that $v_1^*=0$, that is,
\begin{equation}\label{eq:proptoyexam2}
2x_1^*\left(a_4F(\bx^*)-a_1+\frac{a_2a_4}{a_5}\right)=0.
\end{equation}
Using \eqref{mod:toy3obj} and $M=\frac{a_2}{a_5}$, we have
$F(\bx^*)=\frac{\left(a_1-\frac{a_2a_4}{a_5}\right)(x_1^*)^2}{g(\bx^*)}$, and hence
$$
a_4F(\bx^*)-a_1+\frac{a_2a_4}{a_5}=-\left(a_1-\frac{a_2a_4}{a_5}\right)\frac{a_5(x_2^*)^2+a_6}{g(\bx^*)}<0.
$$
This together with \eqref{eq:proptoyexam2} yields $x_1^*=0$. Since $\bx^*\in\Omega_2'$, we have $|x_2^*|\leqs a_0$.
The equality \eqref{mod:toy3obj} also shows that $F(\bx)\geqs0$ for all $\bx\in\Omega_2'$, and that this lower bound is attained at points with $x_1=0$. Since $F(\bx)=+\infty$ outside $\Omega_2'$, $\bx^*$ is a global minimizer of $F$. Equivalently, it is a globally optimal solution of \eqref{mod:toy3}. This derivation also shows that the set of all globally optimal solutions is as stated in this proposition.
\end{proof}

Since $f(\bx)>0$ for all $\bx\in\Omega_2'$, no initial point can satisfy $f(\bx^0)\leqs0$. Nevertheless, Theorem \ref{thm:globalmin} shows that the Sim2-PGA still guarantees convergence to a critical point, and Proposition \ref{prop:toy3global} further shows that this critical point is globally optimal for \eqref{mod:toy3}. This
example therefore demonstrates that $f(\bx^0)\leqs0$ is sufficient, but not necessary, for PGA to converge to a globally optimal solution.

\section{Experimental Results}\label{sec:experiment}
In this section, we conduct experiments based on the examples mentioned in Section \ref{sec:numerexam}. We utilize both simple-example data with ground truths and real-world financial data to evaluate the performance of the proposed method.

\subsection{Experiments for Sim1-PGA}\label{exp:toyexample}
The analytical optimal solution of \eqref{mod:toy} has been determined in Section \ref{subsec:sim1}. In this subsection, we conduct two experiments with ground truth to assess the performance of Sim1-PGA. We define $\bp_A:=(2,-1)^\top$, $\bp_B:=(-2,-1)^\top$, and then set $\bp=\bp_A$ for experiment Sim1-A and $\bp=\bp_B$ for experiment Sim1-B, respectively. From Proposition \ref{prop:toyglobal}, the globally optimal solution of \eqref{mod:toy} is $\bx^*=(0,1)^\top$ for experiment Sim1-A, and $\bx^*=\left(\frac{2}{3},\ \frac{1}{3}\right)^\top$ for experiment Sim1-B. In both experiments, we set $\alpha=0.99/(4\|\bp\|_2)$ and $\bx^0=\left(\frac{1}{2},\frac{1}{2}\right)^\top$. For this common initial point, we have $\bp_A^\top\bx^0=1/2>0$ in Sim1-A and
$\bp_B^\top\bx^0=-3/2<0$ in Sim1-B. Hence, the sufficient initial-point condition
in Theorem \ref{thm:globalmin} is satisfied for Sim1-B but not for Sim1-A. Therefore, Theorem \ref{thm:globalmin} guarantees that the limiting point of the sequence generated by PGA is globally optimal in Sim1-B, while in Sim1-A the theorem only ensures convergence of the generated sequence to a critical point. The experimental results demonstrate that the sequence $\{\bx^k\}_{k\in\bbN}$ generated by Sim1-PGA converges to $\bx^*$ in both cases. Table
\ref{tab_toyexample} (a) and the first column of Figure \ref{fig_toyexample} pertain to experiment Sim1-A, while Table \ref{tab_toyexample} (b) and the second column of Figure \ref{fig_toyexample} pertain to experiment Sim1-B. These results indicate that Sim1-PGA successfully reaches the ground-truth solutions in both cases and reduces the objective function values over the iterations. In particular, Sim1-A shows that the condition $f(\bx^0)\leqs0$ is sufficient for the global optimality guarantee in
Theorem \ref{thm:globalmin}, but is not necessary for PGA to reach a global minimizer in a particular instance.

\begin{table}[htbp]
\centering
\caption{Iterative sequences $\{\bx^k\}_{k\in\bbN}$ generated by Sim1-PGA for Sim1-A and Sim1-B.}
\label{tab_toyexample}
\scalebox{1.1}{
\begin{tabular}{|c|c|c||c|c|c|}
	\hline
\multicolumn{3}{|c||}{(a) Experiment Sim1-A}  & \multicolumn{3}{c|}{(b) Experiment Sim1-B}\\
	\hline
	$k$ & $x_1^k$ & $x_2^k$&   	$k$ & $x_1^k$ & $x_2^k$ \\
    \hline
	0 & 0.5000 & 0.5000 &     	0 & 0.5000 & 0.5000 \\
	\hline
	1 & 0.3340 & 0.6660&   	1 & 0.5553 & 0.4447 \\
	\hline
	2 & 0.1679 & 0.8321&	5 & 0.6427 & 0.3573 \\
	\hline
	3 & 0.0272 & 0.9728 &	10 & 0.6627& 0.3373 \\
	\hline
	4 & 0.0000 & 1.0000 &	20 & 0.6666& 0.3334 \\
	\hline
	5 & 0.0000 & 1.0000 &	27 & 0.6667& 0.3333 \\
	\hline
\end{tabular}
}
\end{table}

\begin{figure}[!htbp]\vspace{2em}
\hspace{6em}\small{Experiment Sim1-A}\hspace{14em}
\small{Experiment Sim1-B}\\
\centering
\subfigure[]{
\includegraphics[width=0.44\textwidth]{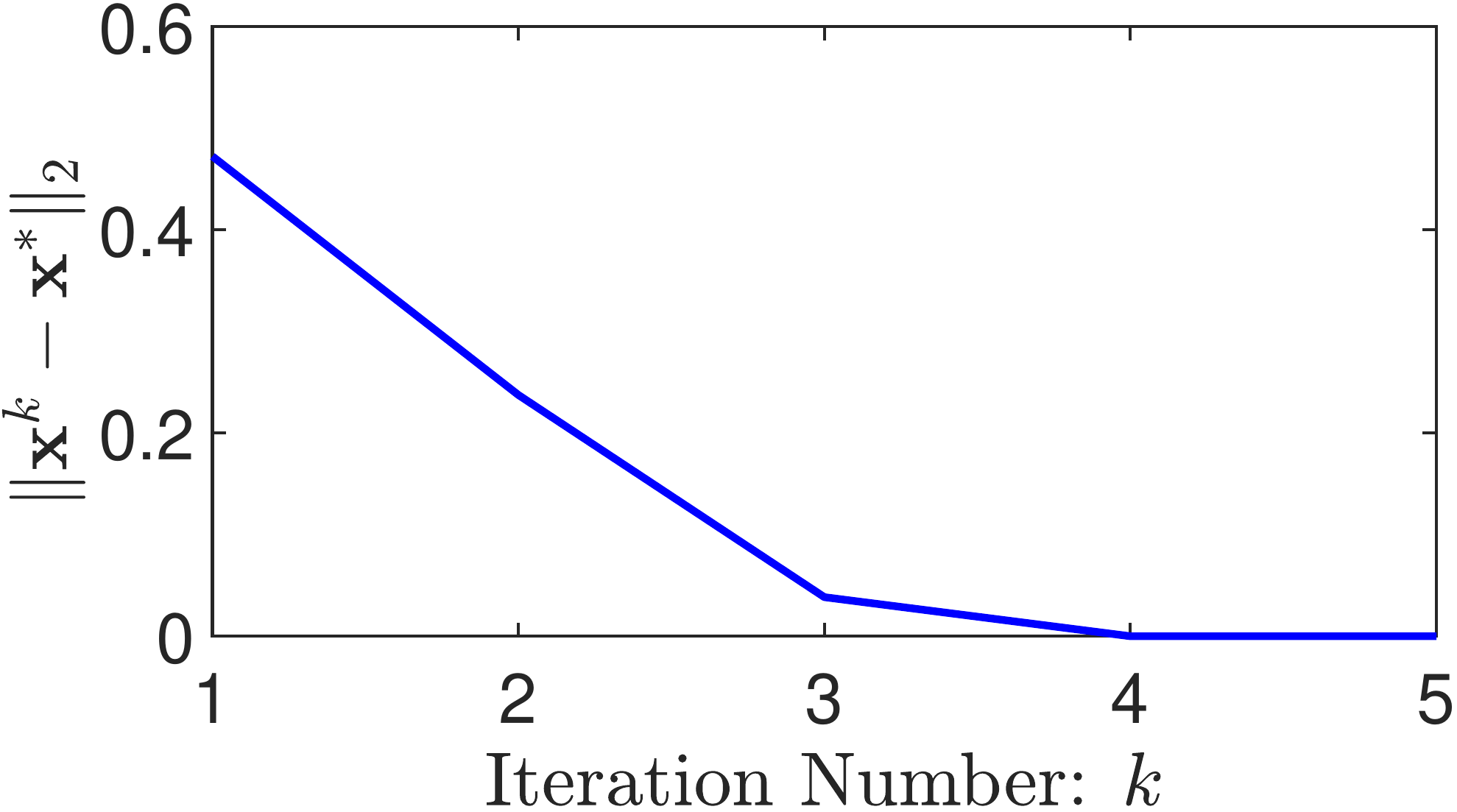}}\hspace{2pt}
\subfigure[]{	\includegraphics[width=0.44\textwidth]{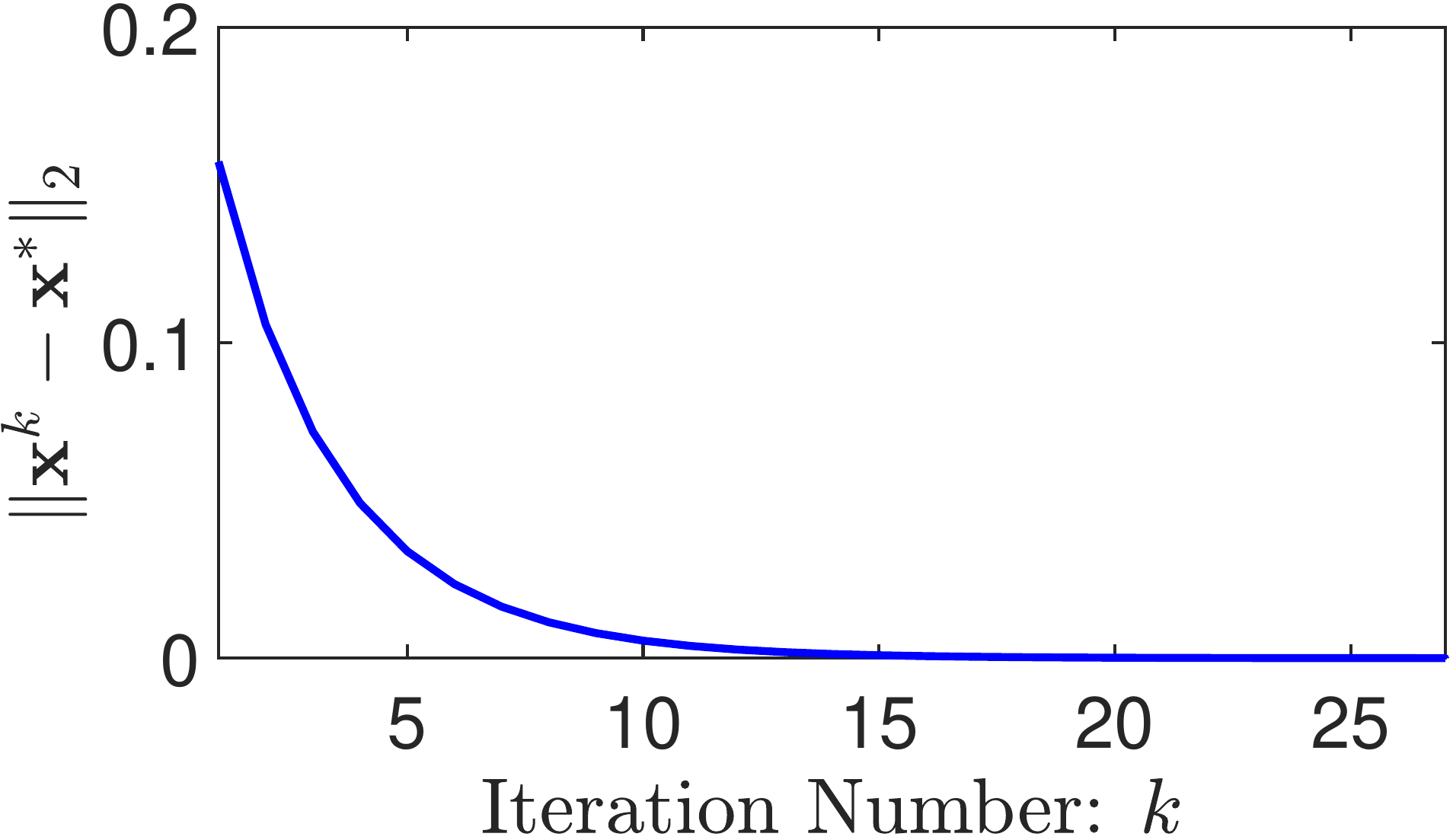}}\\\vspace{-0.5em}
\subfigure[]{ \includegraphics[width=0.44\textwidth]{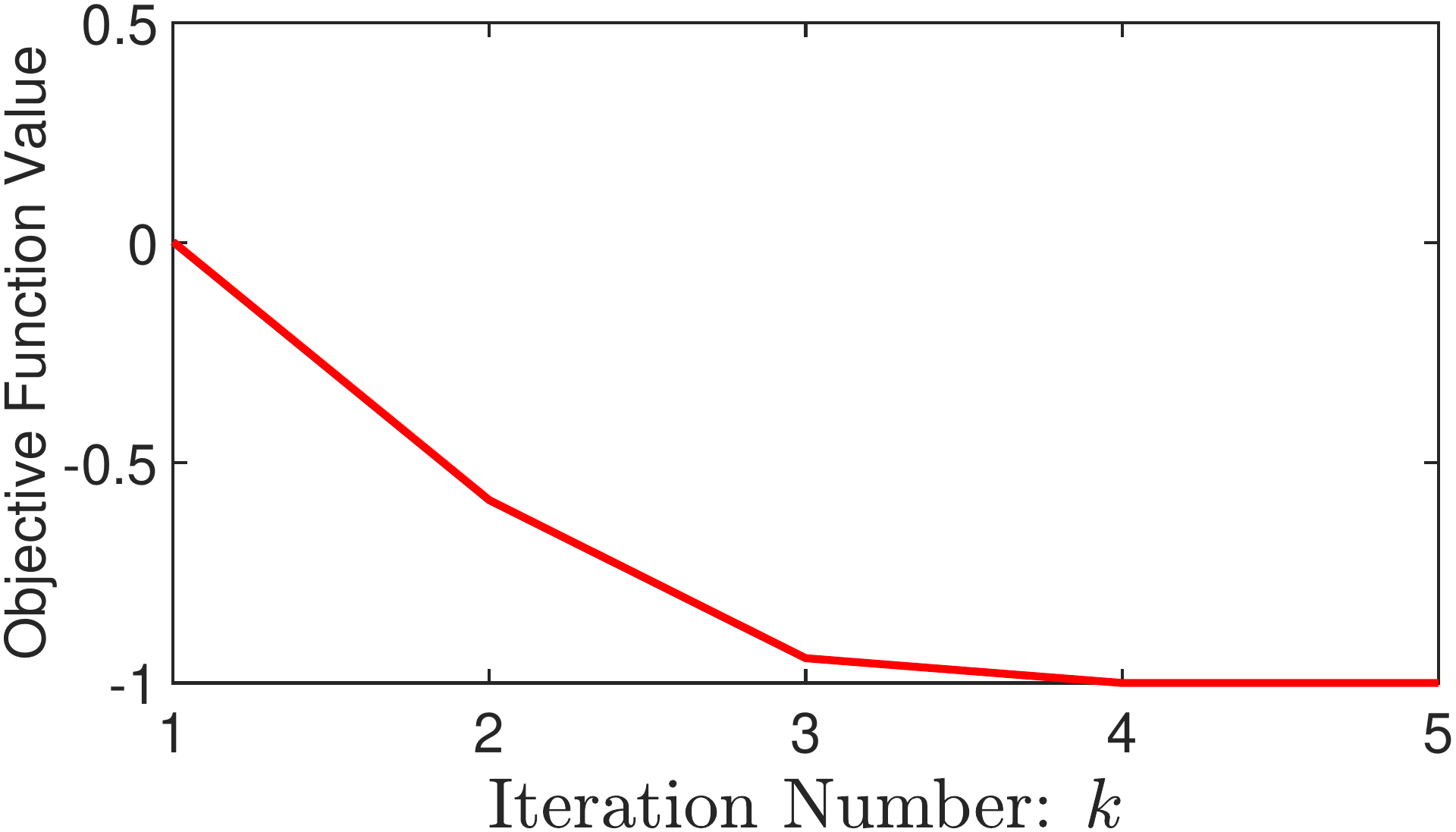}}\hspace{2pt}
\subfigure[]{ \includegraphics[width=0.44\textwidth]{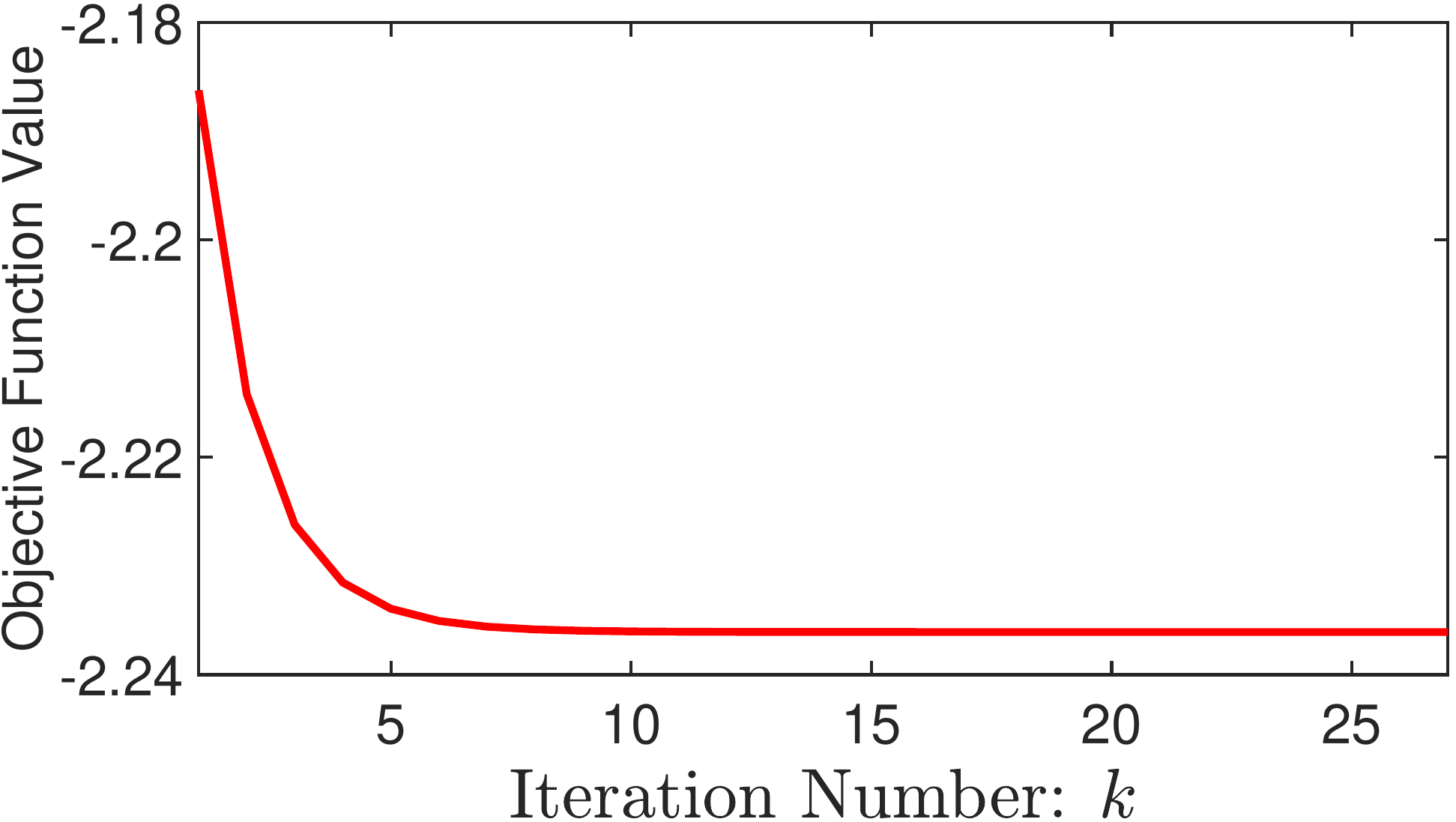}}\\\vspace{-0.5em}
\subfigure[]{	\includegraphics[width=0.44\textwidth]{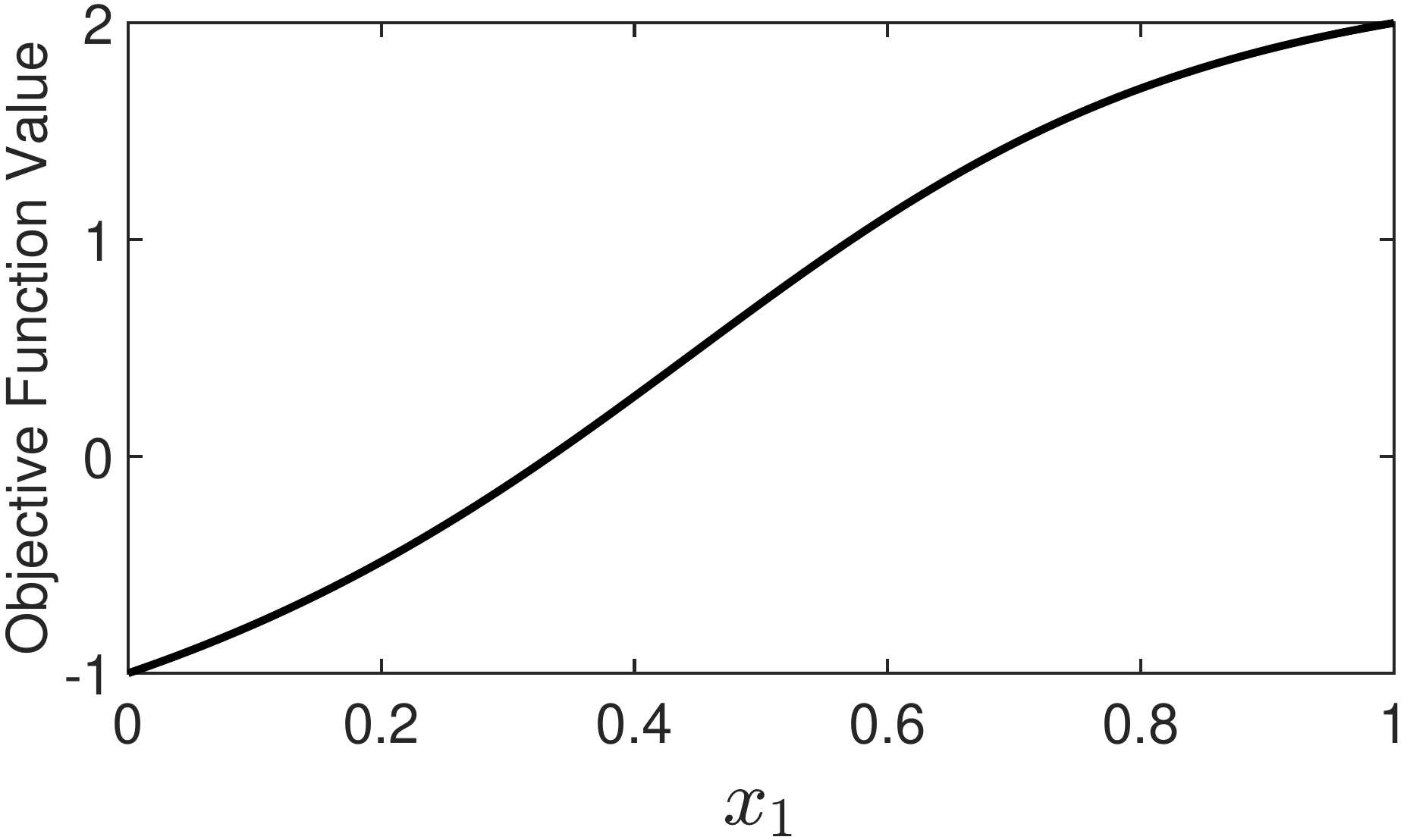}}\hspace{2pt}
\subfigure[]{ \includegraphics[width=0.44\textwidth]{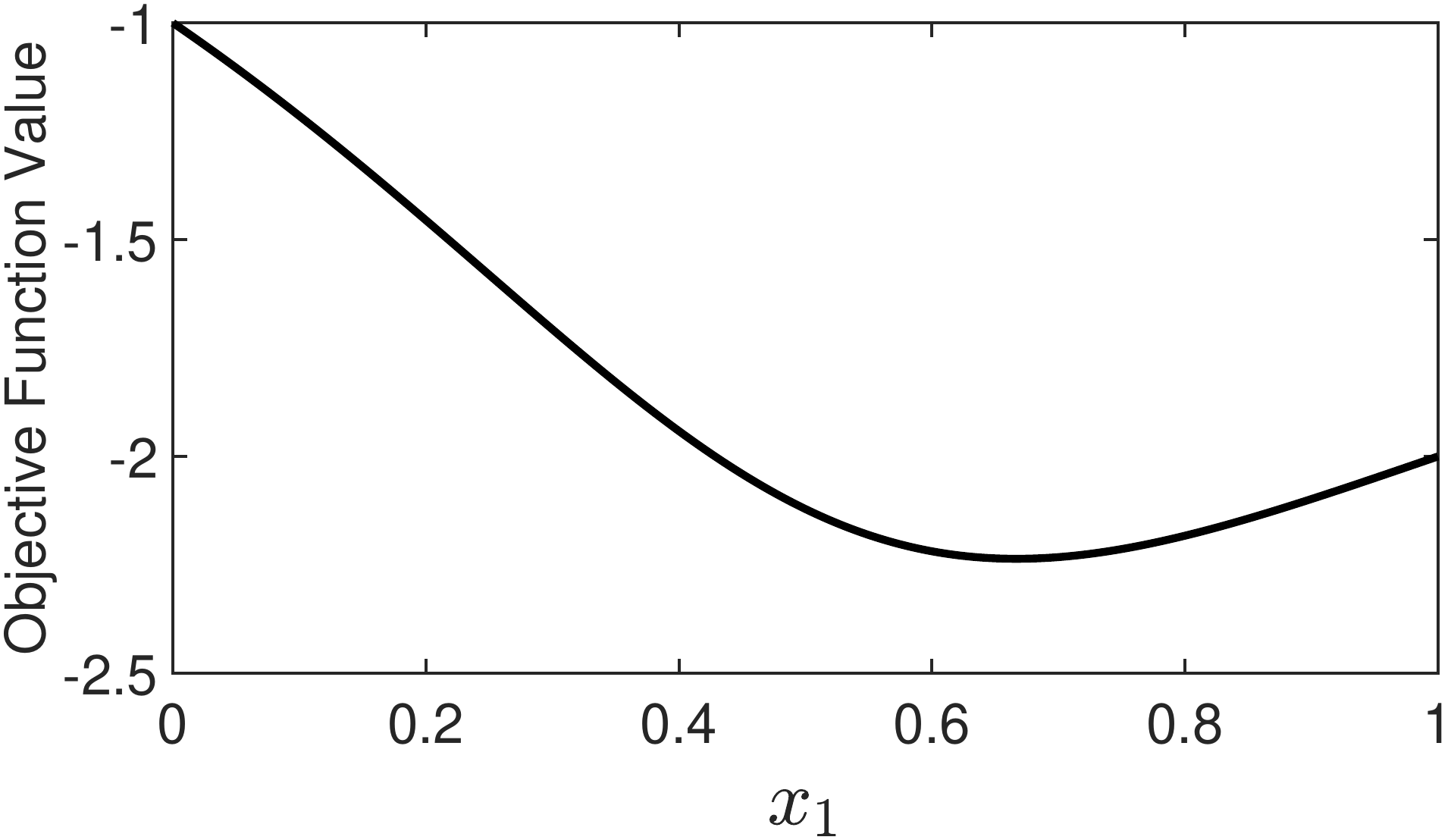}}\vspace{-0.5em}
\caption{Experimental results of Sim1-A and Sim1-B. (a) and (b) suggest that Sim1-PGA converges to the ground-truth solutions. (c) and (d) demonstrate that Sim1-PGA monotonically decreases the objective function until it reaches its minimum. (e) and (f): By substituting $x_2=1-x_1$ into the objective function, it is transformed into a one-variable function in terms of $x_1$.}
\label{fig_toyexample}
\end{figure}

\subsection{Experiments for Sim2-PGA}\label{exp:toyexample2}
In this subsection, we conduct experiments to confirm that Sim2-PGA converges to a globally optimal solution of \eqref{mod:toy3}. We set the parameters in \eqref{mod:toy3} as $a_0=100$, $a_1=4$, $a_2=a_5=2$, and $a_3=a_4=a_6=3$. For the step size and initial point in Sim2-PGA, we define $\alpha=\frac{0.99\cdot a_5}{2(a_1a_5-a_2a_4)}$ and select four different initial points $\bx^0=(50,50)^\top$, $(50,-50)^\top$, $(95,95)^\top$, and $(95,-95)^\top$ for experiments Sim2-A, Sim2-B, Sim2-C, and Sim2-D, respectively. These experiments demonstrate that Sim2-PGA may converge to different globally optimal solutions depending on the initial points used.

In Table \ref{tab_toyexample2}, we illustrate the progression of the iterative sequences as the iteration number $k$ increases for the four distinct initial points. Figure \ref{fig_toyexample2} (a) and (b) display the performance of $\|\bx^k-\bx^*\|_2$ and the objective function value versus the iteration number for Sim2-A, Sim2-B, Sim2-C, and Sim2-D. Figure \ref{fig_toyexample2} (c) depicts the graph of the objective function $\frac{f(x_1,x_2)}{g(x_1,x_2)}:=\frac{4x_1^2+2x_2^2+3}{3x_1^2+2x_2^2+3}$. The experimental results indicate that the sequence $\{\bx^k\}_{k\in\bbN}$ generated by Sim2-PGA indeed converges to a global minimizer of the objective function for each initial point. Furthermore, as the iteration number grows, both the value of $\|\bx^k-\bx^*\|_2$ and the objective function value decrease monotonically.

\begin{table}[htbp]
\centering
\caption{Iterative sequences $\{\bx^k\}_{k\in\bbN}$ generated by Sim2-PGA using four distinct initial points.}
\label{tab_toyexample2}
\setlength{\tabcolsep}{1.5mm}
\scalebox{1.1}{
\begin{tabular}{|c|c|c|| c|c|c|| c|c|c|| c|c|c|}
	\hline
	\multicolumn{3}{|c||}{Experiment Sim2-A}  & \multicolumn{3}{c||}{Experiment Sim2-B}  & \multicolumn{3}{c||}{Experiment Sim2-C}  & \multicolumn{3}{c|}{Experiment Sim2-D}  \\
	\hline
	$k$ & $x_1^k$ & $x_2^k$   &  	$k$ & $x_1^k$ & $x_2^k$  &  $k$ & $x_1^k$ & $x_2^k$   &  $k$ & $x_1^k$ & $x_2^k$ \\
    \hline
	0 & 50.0000 & 50.0000   & 	0 & 50.0000 & -50.0000  &  0 & 95.0000 & 95.0000  & 0 & 95.0000 & -95.0000 \\
	\hline
	1 & 30.1929 & 69.7952  & 	1 & 30.1929 & -69.7952  &  1 & 57.3762 & 100.0000  & 1 & 57.3762 & -100.0000 \\
	\hline
	2 & 6.8518 & 79.8893  & 	2 & 6.8518 & -79.8893   &  2 & 19.3490 & 100.0000   & 2 & 19.3490 & -100.0000 \\
	\hline
	3 & 0.1425 & 80.4646  & 	3 & 0.1425 & -80.4646  &  3 & 1.2119 & 100.0000   & 3 & 1.2119 & -100.0000 \\
	\hline
	4 & 0.0014 & 80.4648  & 	4 & 0.0014 & -80.4648  &  4 & 0.0124 & 100.0000  & 4 & 0.0124 & -100.0000 \\
	\hline
	5 & 0.0000 & 80.4648  & 	5 & 0.0000& -80.4648  & 6 & 0.0000& 100.0000  & 6 & 0.0000& -100.0000 \\
	\hline
\end{tabular}
}
\end{table}

\begin{figure}[htbp]
\centering
\subfigure[]{
\includegraphics[width=0.44\textwidth]{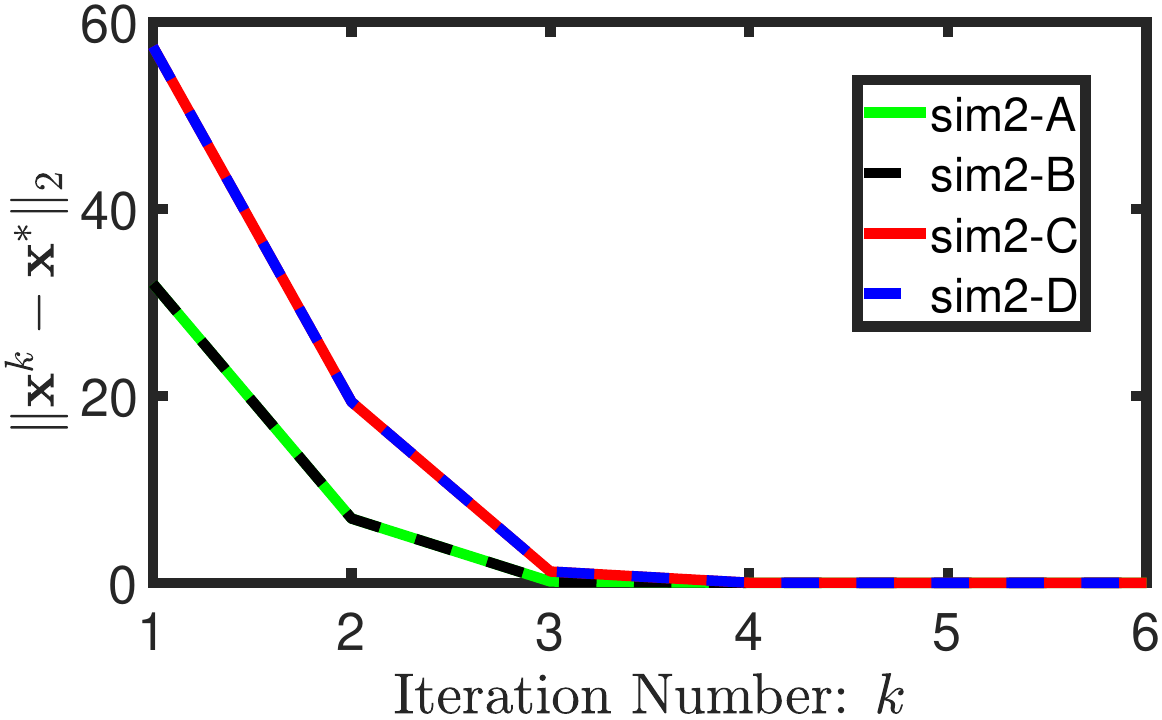}}\hspace{1em}
\subfigure[]{
\includegraphics[width=0.44\textwidth]{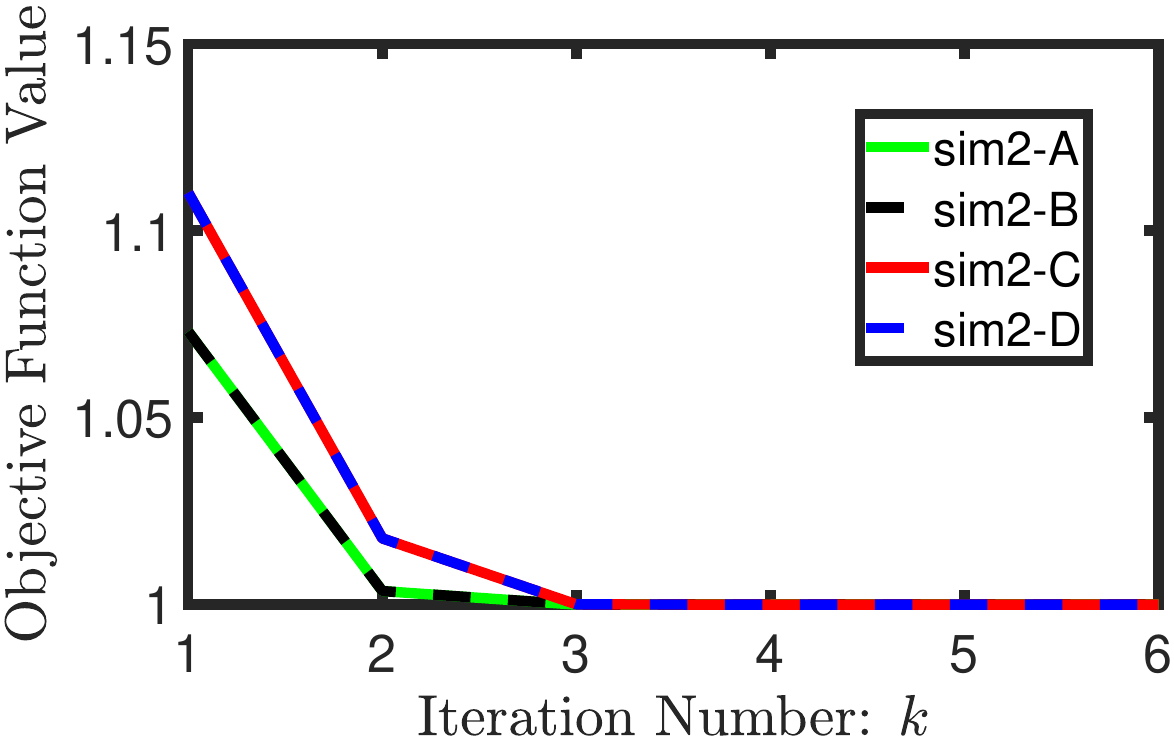}}\\
\subfigure[]{
\includegraphics[width=0.45\textwidth]{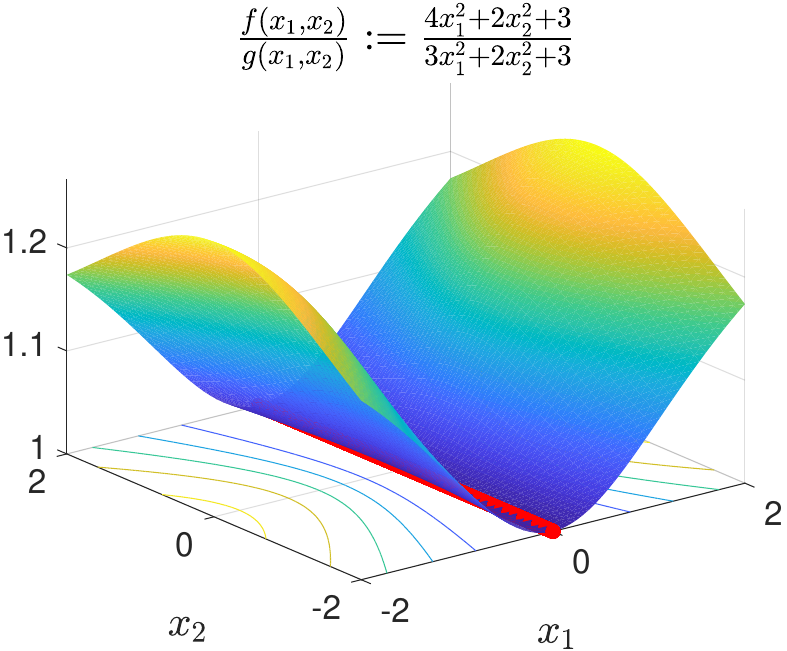}}
\caption{Experimental results of Sim2-A, Sim2-B, Sim2-C, and Sim2-D. (a) confirms that Sim2-PGA converges to a critical point, which is also a global minimizer of the objective function. (b) shows that Sim2-PGA monotonically reduces the objective function to its minimum. (c) presents the graph of the objective function, with its global minimizes indicated by a red line.}
\label{fig_toyexample2}
\end{figure}

\subsection{Real-world data experiments}\label{sec:realexperiment}
In addition to the aforementioned experiments of two simple examples, we also use five real-world monthly benchmark datasets: FRENCH32, FF25EU, FF38, FF49, and FF100MEINV to compare various methods. These datasets are collected from the baseline and commonly used Kenneth R. French's Real-world Data Library\footnote{\url{http://mba.tuck.dartmouth.edu/pages/faculty/ken.french/data_library.html}}. FRENCH32 was constructed by \cite{SPOLC}. FF25EU comprises 25 portfolios formed by ME and prior returns in the European market. FF38 and FF49 consist of 38 and 49 U.S. industry portfolios, respectively. FF100MEINV includes 100 portfolios shaped by ME and investment in the U.S. market. The details of these datasets are presented in Table \ref{tab:infodataset}.

\begin{table}[htbp]
\footnotesize
\centering
\caption{Information of 5 real-world monthly benchmark datasets.}
\label{tab:infodataset}
\begin{tabular}{|c|c|c|c|c|}
\hline
Data Set & Region & Time & Months & Assets\\
\hline
FRENCH32 &US & $Jul/1990\sim Feb/2020$  & 356 &  32\\
FF25EU &EU & $Nov/1990 \sim Oct/2021$ & 372 & 25\\
FF38 &US   & $Jul/1971 \sim Oct/2021$ & 604 & 38\\
FF49 &US   & $Jul/1971 \sim Oct/2021$ & 604 & 49\\
FF100MEINV &US & $Jul/1971 \sim Oct/2021$ & 604 & 100\\
\hline
\end{tabular}
\end{table}

We consider two baseline strategies: the equally-weighted $1/N$ \cite{1Nstrategy} and the Market \cite{olpsjmlr}, along with seven state-of-the-art strategies: Parametric Linear Complementarity Technique (PLCT, \cite{SRlincom}), Sparse and Stable Markowitz Portfolios (SSMP, \cite{sparsepo}), Short-term Sparse Portfolio Optimization (SSPO, \cite{SSPO}), Short-term Portfolio Optimization for Loss Control (SPOLC, \cite{SPOLC}), and Short-term Sparse Portfolio Optimization with $\ell^0$ Regularization (S1, S2 and S3, \cite{SSPOl0}), as competitors in our experiments. The Market and $1/N$ strategies both start with an equally-weighted portfolio; however, the Market remains unchanged while the $1/N$ strategy rebalances to the equally-weighted portfolio at each trading time interval. S1, S2, and S3 are variants of SSPO-$\ell^0$, with S1 being deterministic and S2 and S3 being randomized. The experimental results demonstrate that our proposed SRM-PGA consistently delivers the best performance in most test cases.

We utilize a moving-window trading framework \cite{egrmvgap} for real-world portfolio optimization. Specifically, we use the asset returns $\{\br^{(t)}\}_{t=1}^T$ or the price relatives $\{\mathbf{x}^{(t)}:=\br^{(t)}+\mathbf{1}_N\}_{t=1}^T$ within the time window $t=[1:T]$ to optimize the portfolio $\hat{\bw}^{(T+1)}$ for the subsequent trading period. On the $(T+1)$-th occasion, the portfolio return is calculated by $\hat{r}_{\hat{\bw}}^{(T+1)}=\mathbf{x}^{(T+1)}\cdot \hat{\bw}^{(T+1)}-1$. In the next cycle, the time window shifts to $t=[2:(T+1)]$, and the strategy optimizes the new portfolio $\hat{\bw}^{(T+2)}$. This process continues until the final trading period $\mathscr{T}$, culminating the complete return sequence $\{\hat{r}_{\hat{\bw}}^{(t)}\}_{t=1}^{\mathscr{T}}$. This data can be used for various investment performance evaluation metrics. Initially, when there are insufficient samples to implement a strategy, we default to the equally-weighted portfolio. SSPO, SPOLC, S1, S2, and S3 are short-term strategies that employ a window size of $T=5$. As all other strategies are long-term and do not specify (or involve) $T$, we set $T=20$ for them.

\subsubsection{Results for Sharpe ratios}\label{sec:realexperimentsharpe}
We calculate the backtest SR of a strategy as follows:
\begin{equation}\label{eq:backtestSR}
\widehat{SR}=\frac{\frac{1}{\mathscr{T}}\sum_{t=1}^{\mathscr{T}}\hat{r}_{\hat{\bw}}^{(t)}-r_f}{\sqrt{\frac{1}{\mathscr{T}-1}\sum_{s=1}^{\mathscr{T}}\left(\hat{r}_{\hat{\bw}}^{(s)}-\frac{1}{\mathscr{T}}\sum_{t=1}^{\mathscr{T}}\hat{r}_{\hat{\bw}}^{(t)}\right)^2}},
\end{equation}
where we set $r_f=0$ for the risk-free asset. Table \ref{tab:realresult} displays the (monthly) SRs of the nine strategies compared. The two baseline strategies, $1/N$ and Market, spread risk across all assets, which acts as a natural risk control scheme \cite{1Nstrategy}. Consequently, they significantly outperform most competitors in majority of cases and are comparable to the SR optimizer PLCT. Nonetheless, SRM-PGA attains the highest SRs among all the competitors, including the two baseline strategies, across all datasets. For instance, SRM-PGA's SR is $0.06$ higher than the two baseline strategies on FRENCH32 and $0.02$ higher than PLCT on FF49. Therefore, SRM-PGA not only achieves the optimal solutions in the simple-example experiments but also excels in maximizing the backtest SR with real-world data.

\begin{table}[htbp]
\setlength{\tabcolsep}{1.5mm}
\footnotesize
\centering
\caption{Sharpe ratios of different strategies in real-world data experiments.}
\label{tab:realresult}
\begin{tabular}{|c|c|c|c|c|c|}
	\hline
	Method & FRENCH32 &  FF25EU & FF38 & FF49 & FF100MEINV \\
	\hline
	1/N & 0.2033 & 0.1762 & 0.2205 & 0.2158 & 0.2270 \\
	\hline
	Market & 0.2019 & 0.2458 & 0.2227 & 0.2190 & 0.2351 \\
	\hline
	SPOLC & 0.1215 &  0.0540 & 0.1029 & 0.0948 & 0.1411 \\
	\hline
	SSPO & 0.0685 &  0.0720 & 0.0624 & 0.0506 & 0.1439 \\
	\hline
	S1 & 0.0703 &  0.0830 & 0.0651 & 0.0511 & 0.1507 \\
	\hline
	S2 & 0.0791 &  0.0976 & 0.0653 & 0.0430 & 0.1182 \\
	\hline
	S3 & 0.0956 &  0.1032 & 0.0636 & 0.0547 & 0.1329 \\
	\hline
	SSMP & 0.1445 &  0.1341 & 0.1853 & 0.1778 & 0.1768 \\
	\hline
	PLCT & 0.2539 &  0.2489 & 0.2430 & 0.2371 & 0.2632 \\
	\hline
	\textbf{SRM-PGA} & \textbf{0.2612} &  \textbf{0.2587} & \textbf{0.2609} & \textbf{0.2583} & \textbf{0.2685} \\
	\hline
\end{tabular}
\end{table}

\subsubsection{Results for cumulative wealth}\label{sec:realexperimentcw}
In addition to the backtest SR, portfolio managers are also concerned with the final gain of an investment strategy in actual investment scenarios. By setting the initial wealth for an investment strategy at $W^{(0)}=1$, the manager can compute the final cumulative wealth as $W^{(\mathscr{T})}=\prod_{t=1}^{\mathscr{T}}(\hat{r}_{\hat{\bw}}^{(t)}+1)$. Table \ref{tab:realresultcw} presents the final cumulative wealth results for the nine strategies compared. Generally, the $1/N$, Market, and PLCT strategies demonstrate solid performance across all datasets. However, SRM-PGA achieves the highest final cumulative wealth on five out of six datasets. It narrowly trails PLCT on FRENCH32 but significantly exceeds PLCT on FF38, FF49, and FF100MEINV. For instance, SRM-PGA is approximately $18.47\%$ higher than the second-best competitor, PLCT, on FF100MEINV. Thus, SRM-PGA also proves to be an effective strategy for real-world investment.

\begin{table}[htbp]
\setlength{\tabcolsep}{0.6mm}
\footnotesize
\centering
\caption{Final cumulative wealth of different strategies in real-world data experiments.}
\label{tab:realresultcw}
\begin{tabular}{|c|c|c|c|c|c|c|}
	\hline
	Method & FRENCH32 &  FF25EU & FF38 & FF49 & FF100MEINV \\ \hline
	1/N & 18.32 &  15.77 & 284.43 & 257.93 & 499.29 \\ \hline
	market & 19.34 &  50.01 & 228.52 & 214.15 & 628.43 \\ \hline
	SPOLC & 6.57 &  1.65 & 21.73 & 13.72 & 62.61 \\ \hline
	SSPO & 2.19 &  2.57 & 1.77 & 0.80 & 122.03 \\ \hline
	S1 & 2.29 &  3.32 & 2.08 & 0.84 & 152.24 \\ \hline
	S2 & 3.06 &  5.06 & 2.91 & 0.60 & 48.12 \\ \hline
	S3 & 4.21 &  5.28 & 2.14 & 1.27 & 62.64 \\ \hline
	SSMP & 11.78 & 8.47 & 192.44 & 167.23 & 513.81 \\ \hline
	PLCT & \textbf{52.62} &  71.50 & 978.48 & 797.68 & 911.32 \\ \hline
	\textbf{SRM-PGA} & 50.15 &  \textbf{73.61} & \textbf{1061.66} & \textbf{877.09} & \textbf{1079.65}\\ \hline
\end{tabular}
\end{table}

\section{Conclusion}\label{sec:conclusion}
In this work, we derive globally optimal solutions for a class of fractional optimization problems, predicated on certain assumptions regarding the objective function and constraint set. We employ the proximal gradient technique to solve the fractional optimization model and establish the convergence of the solution algorithm to a critical point of the objective function. Under the verifiable initial-point condition \(f(\bx^0)\leqs0\), this limiting critical point is further shown to be a globally optimal solution. We successfully apply these results to the problem of maximizing the Sharpe ratio.

Experiments using simple-example data with known optima reveal that our method consistently attains these known optima and reduces the objective function across iterations. Additionally, experiments with real-world financial datasets demonstrate that our method consistently achieves the highest backtested Sharpe ratios and the highest final cumulative wealths in most scenarios, compared to two baseline strategies and seven advanced competitors. In conclusion, our proposed method shows promise in fractional optimization and merits further exploration.

\section*{Data availability statement}
The datasets analysed during the current study are available in Fama and French's Data Library,
\url{https://mba.tuck.dartmouth.edu/pages/faculty/ken.french/data_library.html}


%
%


\bibliographystyle{spmpsci}
\bibliography{bibfile}

\begin{thebibliography}{10}
\providecommand{\url}[1]{{#1}}
\providecommand{\urlprefix}{URL }
\expandafter\ifx\csname urlstyle\endcsname\relax
  \providecommand{\doi}[1]{DOI~\discretionary{}{}{}#1}\else
  \providecommand{\doi}{DOI~\discretionary{}{}{}\begingroup
  \urlstyle{rm}\Url}\fi

\bibitem{attouch2009convergence}
Attouch, H., Bolte, J.: On the convergence of the proximal algorithm for
  nonsmooth functions involving analytic features.
\newblock Mathematical Programming \textbf{116}, 5--16 (2009)

\bibitem{attouch2010proximal}
Attouch, H., Bolte, J., Redont, P., Soubeyran, A.: Proximal alternating
  minimization and projection methods for nonconvex problems: An approach based
  on the {Kurdyka-{\L}ojasiewicz} inequality.
\newblock Mathematics of Operations Research \textbf{35}(2), 438--457 (2010)

\bibitem{bauschke2017convex}
Bauschke, H.H., Combettes, P.L.: Convex Analysis and Monotone Operator Theory
  in {Hilbert} Space, 2nd edn.
\newblock Springer, New York (2017)

\bibitem{bolte2014proximal}
Bolte, J., Sabach, S., Teboulle, M.: Proximal alternating linearized
  minimization for nonconvex and nonsmooth problems.
\newblock Mathematical Programming \textbf{146}(1), 459--494 (2014)

\bibitem{boct2017proximal}
Bo{\c{t}}, R.I., Csetnek, E.R.: Proximal-gradient algorithms for fractional
  programming.
\newblock Optimization \textbf{66}(8), 1383--1396 (2017)

\bibitem{boct2022extrapolated}
Bo{\c{t}}, R.I., Dao, M.N., Li, G.: Extrapolated proximal subgradient
  algorithms for nonconvex and nonsmooth fractional programs.
\newblock Mathematics of Operations Research \textbf{47}(3), 2415--2443 (2022)

\bibitem{boct2023inertial}
Bo{\c{t}}, R.I., Dao, M.N., Li, G.: Inertial proximal block coordinate method
  for a class of nonsmooth sum-of-ratios optimization problems.
\newblock SIAM Journal on Optimization \textbf{33}(2), 361--393 (2023)

\bibitem{sparsepo}
Brodie, J., Daubechies, I., Mol, C.D., Giannone, D., Loris, I.: Sparse and
  stable {M}arkowitz portfolios.
\newblock Proceedings of the National Academy of Sciences of the United States
  of America \textbf{106}(30), 12267--12272 (2009)

\bibitem{FPapp7}
Cooper, W.W., Seiford, L.M., Tone, K., Zhu, J.: Some models and measures for
  evaluating performances with {DEA}: past accomplishments and future
  prospects.
\newblock Journal of Productivity Analysis \textbf{28}(3), 151--163 (2007)

\bibitem{coste2002introduction}
Coste, M.: An introduction to semialgebraic geometry.
\newblock RAAG Notes, Institut de Recherche Math{\'e}matique de Rennes pp.
  1--78 (2002).
\newblock
  \urlprefix\url{https://gcomte.perso.math.cnrs.fr/M2/CosteIntroToSemialGeo.pdf}

\bibitem{SRprinpivot}
Cottle, R.W.: Monotone solutions of the parametric linear complementarity
  problem.
\newblock Mathematical Programming \textbf{3}(1), 210--224 (1972)

\bibitem{FPapp2}
Cottle, R.W., Infanger, G.: Harry {Markowitz} and the early history of
  quadratic programming.
\newblock Handbook of Portfolio Construction pp. 179--211 (2010)

\bibitem{FPreview1}
Crouzeix, J.P., Ferland, J.A.: Algorithms for generalized fractional
  programming.
\newblock Mathematical Programming \textbf{52}(1), 191--207 (1991)

\bibitem{1Nstrategy}
DeMiguel, V., Garlappi, L., Uppal, R.: Optimal versus naive diversification:
  How inefficient is the {1/N} portfolio strategy?
\newblock The Review of Financial Studies \textbf{22}(5), 1915--1953 (2009)

\bibitem{frdindelbach}
Dinkelbach, W.: On nonlinear fractional programming.
\newblock Management Science \textbf{13}(7), 492--498 (1967)

\bibitem{simplexproject}
Duchi, J., Shalev-Shwartz, S., Singer, Y., Chandra, T.: Efficient projections
  onto the $\ell_1$-ball for learning in high dimensions.
\newblock In: Proceedings of the International Conference on Machine Learning
  (ICML) (2008)

\bibitem{FPreview3}
Frenk, J.B.G., Schaible, S.: Handbook of Generalized Convexity and Generalized
  Monotonicity, 2nd edn., chap. Fractional Programming, pp. 335--386.
\newblock Springer, New York, NY (2005)

\bibitem{FPapp1}
Jones, C.K.: Digital portfolio theory.
\newblock Computational Economics \textbf{18}, 287--316 (2001)

\bibitem{SPOLC}
Lai, Z.R., Tan, L., Wu, X., Fang, L.: Loss control with rank-one covariance
  estimate for short-term portfolio optimization.
\newblock Journal of Machine Learning Research \textbf{21}(97), 1--37 (2020).
\newblock \urlprefix\url{http://jmlr.org/papers/v21/19-959.html}

\bibitem{egrmvgap}
Lai, Z.R., Yang, H.: A survey on gaps between mean-variance approach and
  exponential growth rate approach for portfolio optimization.
\newblock ACM Computing Surveys \textbf{55}(2), 1--36 (2023).
\newblock \doi{https://doi.org/10.1145/3485274}.
\newblock {Article No. 25}

\bibitem{SSPO}
Lai, Z.R., Yang, P.Y., Fang, L., Wu, X.: Short-term sparse portfolio
  optimization based on alternating direction method of multipliers.
\newblock Journal of Machine Learning Research \textbf{19}(63), 1--28 (2018).
\newblock \urlprefix\url{http://jmlr.org/papers/v19/17-558.html}

\bibitem{olpsjmlr}
Li, B., Sahoo, D., Hoi, S.C.: {OLPS}: a toolbox for on-line portfolio
  selection.
\newblock Journal of Machine Learning Research \textbf{17}(1), 1242--1246
  (2016)

\bibitem{li2022proximal}
Li, Q., Shen, L., Zhang, N., Zhou, J.: A proximal algorithm with backtracked
  extrapolation for a class of structured fractional programming.
\newblock Applied and Computational Harmonic Analysis \textbf{56}, 98--122
  (2022)

\bibitem{SSPOl0}
Luo, Z., Yu, X., Xiu, N., Wang, X.: Closed-form solutions for short-term sparse
  portfolio optimization.
\newblock Optimization  (2020)

\bibitem{mordukhovich2018variational}
Mordukhovich, B.S.: Variational Analysis and Applications, vol.~30.
\newblock Springer (2018)

\bibitem{SRlincom}
Pang, J.S.: A parametric linear complementarity technique for optimal portfolio
  selection with a risk-free asset.
\newblock Operations Research \textbf{28}(4), 927--941 (1980)

\bibitem{frdindelbach2}
Pardalos, P.M., Phillips, A.T.: Global optimization of fractional programs.
\newblock Journal of Global Optimization \textbf{1}(2), 173--182 (1991)

\bibitem{FPreview2}
Schaible, S.: Fractional programming. {II}, on {D}inkelbach's algorithm.
\newblock Management Science \textbf{22}(8), 868--873 (1976)

\bibitem{FPreview4}
Schaible, S., Ibaraki, T.: Fractional programming.
\newblock European Journal of Operational Research \textbf{12}(4), 325--338
  (1983)

\bibitem{FPapp3}
Shen, K., Yu, W.: Fractional programming for communication systems—part {I}:
  Power control and beamforming.
\newblock IEEE Transactions on Signal Processing \textbf{66}(10), 2616--2630
  (2018)

\bibitem{FPapp6}
Zappone, A., Björnson, E., Sanguinetti, L., Jorswieck, E.: Globally optimal
  energy-efficient power control and receiver design in wireless networks.
\newblock IEEE Transactions on Signal Processing \textbf{65}(11), 2844--2859
  (2017)

\bibitem{FPapp4}
Zeng, Y., Zhang, R.: Energy-efficient {UAV} communication with trajectory
  optimization.
\newblock IEEE Transactions on Wireless Communications \textbf{16}(6),
  3747--3760 (2017)

\bibitem{FPapp5}
Zhang, H., Duan, Y., Long, K., Leung, V.C.M.: Energy efficient resource
  allocation in terahertz downlink {NOMA} systems.
\newblock IEEE Transactions on Communications \textbf{69}(2), 1375--1384 (2021)

\bibitem{zhang2022first}
Zhang, N., Li, Q.: First-order algorithms for a class of fractional
  optimization problems.
\newblock SIAM Journal on Optimization \textbf{32}(1), 100--129 (2022)

\end{thebibliography}

%

\end{document}